\documentclass[preprint]{elsarticle}

\usepackage[utf8]{inputenc}
\usepackage{amsmath}
\usepackage{hyperref}
\usepackage{amssymb}
\usepackage{amsthm}
\usepackage{graphicx}
\usepackage{amsmath}
\usepackage{amscd}
\usepackage{amsfonts}
\usepackage{multicol}
\usepackage{multirow}
\usepackage{url}
\usepackage{mathtools}
\usepackage{tikz}
\usetikzlibrary{matrix,decorations.pathreplacing,calc}
\usepackage{pdflscape}
\usepackage{rotating}
\usepackage{tikz}

\usepackage{pgfplots}

\usetikzlibrary{arrows}

\usetikzlibrary{automata}

\usetikzlibrary{chains}

\usetikzlibrary{positioning}
\usetikzlibrary{backgrounds}
\usetikzlibrary{calc,decorations.pathreplacing}
\usepackage{amssymb,amsmath}

\pgfplotsset{compat=1.9}

\makeatletter 
\g@addto@macro\@floatboxreset{\centering} 
\makeatother

\makeatother
\pgfkeys{tikz/mymatrixenv/.style={decoration=brace,every left delimiter/.style={xshift=3pt},every right delimiter/.style={xshift=-3pt}}}
\pgfkeys{tikz/mymatrix/.style={matrix of math nodes,left delimiter=[,right delimiter={]},inner sep=2pt,column sep=1em,row sep=0.5em,nodes={inner sep=0pt}}}
\pgfkeys{tikz/mymatrixbrace/.style={decorate,thick}}

\DeclareMathOperator{\diag}{diag}

\newtheorem{teo}{{Theorem}}[section]
\newtheorem{cor}[teo]{Corollary}
\newtheorem{lema}[teo]{Lemma}
\newtheorem{prop}[teo]{Proposition}
\theoremstyle{definition}
\newtheorem{defi}[teo]{Definition}
\theoremstyle{remark}
\newtheorem{remark}[teo]{Remark}
\newtheorem{example}[teo]{Example}
\DeclareMathOperator{\Gal}{Gal}
\DeclareMathOperator{\GL}{GL}

\DeclareMathOperator{\id}{id}

\DeclareMathOperator{\End}{End}

\DeclareMathOperator{\Rep}{Rep}
\DeclareMathOperator{\reg}{reg}
\DeclareMathOperator{\M}{Mat}
\DeclareMathOperator{\Hom}{Hom}
\DeclareMathOperator{\op}{op}
\DeclareMathOperator{\tw}{tw}
\DeclareMathOperator{\cogrado}{codeg}
\DeclareMathOperator{\modu}{mod}

\newcommand{\calm}{\ensuremath{\mathcal{M}}}
\newcommand{\caln}{\ensuremath{\mathcal{N}}}

\newcommand{\C}{\mathbb{C}}
\newcommand{\Z}{\mathbb{Z}}
\newcommand{\N}{\mathbb{N}}
\newcommand{\F}{\mathbb{F}}
\newcommand{\A}{\alpha}
\newcommand{\B}{\beta}
\newcommand{\G}{\gamma}
\newcommand{\D}{\delta}
\newcommand{\E}{\varepsilon}
\newcommand{\Ep}{\epsilon}

\newcommand{\Si}{\sigma}

\newcommand{\mcm}{\ensuremath{\mathrm{lcm}}}
\newcommand{\mcd}{\ensuremath{\mathrm{gcd}}}
\newcommand{\ima}{\ensuremath{\mathrm{Im}}}

\newcommand{\spec}{\ensuremath{\mathrm{Spec}}}

\newcommand{\lra}{\ensuremath{\rightarrow}}

\newcommand{\parE}{(\!(\E)\!)}
\newcommand{\parEn}{(\!(\E^\frac{1}{n})\!)}

\newcommand{\parEc}{(\!(\E^\frac{1}{4})\!)}
\newcommand{\lp}{(\!(}
\newcommand{\rp}{)\!)}

\journal{arXiv}
\begin{document}
\begin{frontmatter}{}
\title{A model for the canonical algebras of bimodules type (1, 4) over truncated polynomial rings.}
\author[rtv]{Christof Geiss\fnref{fn1}}
\ead{christof.geiss@im.unam.mx}
\author[rtv]{David Reynoso-Mercado\fnref{fn2}}
\ead{davreynosom@gmail.com}
\fntext[fn1]{Instituto de Matem\'aticas, Universidad Nacional Aut\'onoma de M\'exico, Ciudad Universitaria, 04510 Ciudad M\'exico, M\'exico.}
\fntext[fn2]{\'Algebra, Teor\'ia de N\'umeros y Aplicaciones: ERM (ALTENUA) y \'Algebra UdeA, Instituto de Matem\'aticas, Facultad de Ciencias Exactas y Naturales, Universidad de Antioquia UdeA, Calle 70 No. 52-21, Medell\'in, Colombia.}

\begin{abstract}
Let $\C\parE$ be the field of complex Laurent series. We use Galois descent techniques to show that the simple regular representations of the species of type $(1,\, 4)$ over $\C\parE$ are naturally parametrized by the closed points of $\spec(\C\parE[x])\dot{\cup}\{1,\,2\}$. Moreover, we provide weak normal forms for those representations. We use our representatives of the simple regular representations to describe the canonical algebras associated to the species of type (1, 4) over $\C\parE$. This suggests a model of those algebras in the sense of the work of Geiss, Leclerc and Schr\"oer \cite{GLS16} and \cite{GLS20}. 
\end{abstract}

\begin{keyword}
Galois descent \sep Canonical algebra \sep Quivers with automorphisms \sep regular representations.

%% PACS codes here, in the form: \PACS code \sep code

\MSC[2020] 16G20 \sep 16S35 
%% or \MSC[2008] code \sep code (2000 is the default)

{\bf Corresponding Author:} David Reynoso-Mercado, e-mail: davreynosom@gmail.com
\end{keyword}
\end{frontmatter}{}
\tableofcontents

%\maketitle

\section{Introduction}

Canonical algebras were introduced by Ringel~\cite{Ri90} in the context of tame algebras.
They play a prominent role in representation theory. Happel showed in~\cite{Ha01} that every $\F$-linear hereditary abelian category with certain finiteness conditions  is derived equivalent to the category of representations  of a canonical algebra (defined in terms of a quiver with relations) when $\F$ is an algebraically closed field. Later, Happel and Reiten were able to lift the restriction of $\F$ being an algebraically closed field in \cite{HR02} by working with species and canonical algebras in the  sense of Ringel~\cite{Ri90}. It is worth mentioning that canonical algebras, almost since their introduction, have been closely related to the study of sheaves on weighted projective lines. The case of algebraically closed fields is explored in~\cite{GL89}, while the general case is analyzed in \cite{Ku09}.

Let $\F$ be a field. In \cite{GLS16}, a new class of $\F$-algebras, $H_\F(C,\,D,\,\Omega)$, was introduced. This allows, to some extent, to model representations of species as locally free modules over these algebras. These algebras are defined in terms of quivers with relations and are a generalization of the classical tensor algebras/species associated with modulated graphs.

The aim of this work is to initiate the study of
a similar model, in terms of quivers with relations, for canonical algebras of possibly non-simply laced types.

As we will see in  Section \ref{sec4.1}, following  Ringel's definitions from~\cite{Ri90}, there are two types of canonical algebras to consider: those of bimodule type $(1,\,4)$ and those of bimodule type $(2,\,2)$. For the purposes of this paper, we will only examine the canonical algebras of bimodule type $(1,\,4)$. The canonical algebras of bimodules type $(2,\,2)$  are examined in \cite{GRR23}. Inspired 
by ~\cite{GLS20}, we need to find explicit descriptions for the canonical algebras of this type over the field of complex Laurent series $\C\parE$, and then reduce an integral form of these algebras modulo $\E^k$. To find the above mentioned explicit description, we first need to find normal forms for the simple regular representations of the species $\Lambda$ of type $(1,\,4)$ ($\tilde{A}_{12}$ in the notation of \cite{DR75})
over the field $\C\parE$.
Since this field is  quasi-finite, we can use Galois descent for our purpose.
\medskip

More concretely $\Lambda$ can be described as the
matrix algebra 
\[
\Lambda\coloneqq\left[\begin{array}{cc}K&K\\0&k\end{array}\right], 
\]
where $K$ is defined as $\C\parEc$ and $k$ is defined as $\C\parE$.  

In Section \ref{sec2}, we first recall some
consequences from the fact that $\C\parE$ is a quasi-finite field. In particular, we discuss the
finite extensions $k_n$ and their Galois groups.
This allows us to relate the simple regular representations of $\Lambda$ with simple regular
$k_n$-linear representations of a quiver of type
$\widetilde{\mathsf{D}}_4$. After introducing a
bit more of notation, we provide in 
Theorem~\ref{teorema1.5} explicitly weak normal forms for the simple regular representations of $\Lambda$. For
this purpose in particular, we need to introduce the notion of a~\emph{regular matrix}
for the elements of $\M_{n\times n}(K)$,
see Definition~\ref{def:regular}.

In Section \ref{sec3}, we discuss in more detail
the technique of Galois descent 
and its relation to quivers with automorphisms, following Hubery~\cite{Hu02, Hu04}. As an application, we study the species of bimodule type $(1, 4)$ over our field of Laurent series. This section presents necessary technical tools for proving our main results.

One of our main results is the following, which 
is our main tool for the classification of
the simple regular $\Lambda$-modules (see Theorem \ref{lema1}).
\begin{teo}
For each $n\in \Z_{>0}$ and $R$ a regular matrix:
\begin{enumerate}
\item The representation $M(R)$ of $\Lambda$ is simple regular if and only if $\phi_n(R)\in k_n$ is generic. In this case, $\End_\Lambda(M(R))\cong k_n$. 
%If $R,\,R'\in K[\A_n]_{\reg}$, then $M(R)\cong M(R')$ if and only if $\phi_n(R)=\Si_n^m(\phi_n(R'))$ for some $1\leq m\leq n-1$.

\item The representations $M_{(1)}$  and $M_{(2)}$ are simple regular with $\End_\Lambda(M_{(1)})\cong k_{2}$ and $\End_\Lambda(M_{(2)})\cong k_{4}$. 
\item Each simple regular representation of $\Lambda$ is isomorphic to  one of the types in {\em a)} or {\em b)}.
\end{enumerate}
\end{teo}
Where $M_{(1)}$ and $M_{(2)}$ are the representations related to the matrices $[1,\,\E^\frac{1}{2}]$ and $[\id_2|\A_2\E^\frac{1}{4}-i\id_2\E^\frac{3}{4}]$, respectively. 

Finally, in Section \ref{sec4}, we present Ringel's definition of canonical algebra and classify all simple regular modules over the species of type (1, 4) over $k$. We also introduce generalized canonical algebras in terms of quivers with relations. In fact, these algebras are truncations of an integral form of Ringel's canonical algebras over a field of Laurent series.  
In other words, we combine Ringel's definition of the species version of canonical algebras with ideas from \cite{GLS16} (used to define $H(C,\,D,\,\Omega)$), see subsection \ref{3.4}.
\begin{example}
    Consider the following quiver
    \begin{center}
\begin{tikzpicture}
[->,>=stealth',shorten >=1pt,auto,node distance=1.5cm,thick,main node/.style=]
  \node[main node] (1) {$2$};   
 \node[main node] (5) [below left of=1] {};
  \node[main node] (3) [left of=5] {$1$};
  \node[main node] (4) [below right of=3] {$3$};
  \node[main node] (6) [ right of=4] {};
  \node[main node] (12) [right  of=6]{$4$};   
\node[main node] (11) [above right  of=12]{$5$};   
\path[every node/.style={font=\sffamily\small}]     
(1) edge node   {$\A_{12}$} (3)
(1) edge [loop above] node              {$\epsilon_2$} (1)
(3) edge [loop above] node              {$\epsilon_1$} (3)
(11) edge [loop above] node              {$\epsilon_5$} (11)
(4) edge [loop below] node              {$\epsilon_3$} (4)
(12) edge [loop below] node              {$\epsilon_4$} (12)
 (11) edge node         {$\A_{25}$} (1)             
  (12) edge node         {$\A_{34}$} (4)             
 (11) edge [bend right=10] node  [left]{$\A_{45}^{(1)}$} (12)        
(11) edge [bend left=10] node     {$\A_{45}^{(2)}$} (12)
 (4) edge node  {$\A_{13}$} (3);
\end{tikzpicture}
\end{center}
with the relations
$$\epsilon_1^4=\epsilon_2^2=\epsilon_3^3=\epsilon_4^3=\epsilon_5=0,$$ 
and 

\begin{minipage}{4cm}
\begin{eqnarray*}
\epsilon_3\A_{34}&=&\A_{34}\epsilon_4,\\\epsilon_4^3\A^{(1)}_{45}&=&\A^{(1)}_{45}\epsilon_5,\\\epsilon_4^3\A^{(2)}_{45}&=&\A^{(2)}_{45}\epsilon_5,\\
\epsilon_1^2\A_{12}&=&-\A_{12}\epsilon_2,
\end{eqnarray*}
\end{minipage}
\begin{minipage}{5cm}
\begin{eqnarray*}
\A_{13}\epsilon_3^2\A_{34}\A_{45}^{(1)}&=&-e^\frac{\pi i}{4}\epsilon_1^3\A_{12}\A_{25},\\\epsilon_1^4\A_{13}&=&\A_{13}\epsilon_3^3,\\\A_{13}\epsilon_3^2\A_{34}\A_{45}^{(2)}&=&\A_{12}\A_{25},\\
\A_{13}\A_{34}\A_{45}^{(1)}&=&-\epsilon_1^2\A_{12}\A_{25},
\end{eqnarray*}
\end{minipage}
\begin{minipage}{4.1cm}
\begin{eqnarray*}
\epsilon_2^2\A_{25}&=&\A_{25}\epsilon_5,\\
\A_{13}\A_{34}\A_{45}^{(2)}&=&0,\\\A_{13}\epsilon_3\A_{34}\A_{45}^{(1)}&=&0,\\\A_{13}\epsilon_3^2\A_{34}\A_{45}^{(1)}&=&0.\end{eqnarray*}
\end{minipage}
 We will see in subsection \ref{3.4} that this quiver with relations is our model of  the canonical algebra associated to $\Lambda$, where the list of simple regular representations consists of $M_{(1)}$ and $M$, while the corresponding branch lengths are 2 and 3, respectively.

Here, the matrix of $M$ is $\left[\begin{array}{cccccc}1&0&0&e^{\frac{\pi i}{4}}\E^\frac{3}{4}&0&\E^\frac{1}{2}\\0&1&0&\E^\frac{3}{2}&e^{\frac{\pi i}{4}}\E^\frac{3}{4}&0\\0&0&1&0&\E^\frac{3}{2}&e^{\frac{\pi i}{4}}\E^\frac{3}{4}\end{array}\right]$. 
\end{example}

For unexplained notation and definitions we refer to \cite{DR75,GLS16,RR85,Ri90}.
%%%%%%%%%%%%%%%%%%%%%%%%%%%%%%%%%%%%%%%%%%%%%%%%%%%%%%%%%%%

\section{A first description of simple regular representations of \texorpdfstring{$\Lambda$}{Λ}}
\label{sec2}
\subsection{ \texorpdfstring{$\C\parE$}{C((ε))} as a quasi-finite field}\label{quasi1.2}
 Recall from \cite[Chapter XIII, Section 2]{S79} that $k=\C\parE$ is a quasi-finite field because $\C$ is an algebraically closed field of characteristic $0$. Thus, every finite dimensional division algebra over $k$ is isomorphic to  $k_n\coloneqq \C\parEn$ for some $n\in\{1,\,2,\,3,\cdots\}=\Z_{>0}$. In particular, all finite dimensional division algebras are commutative, and  the Galois group of $k_n/k$ is the cyclic group $C_n=\langle\Si_n\rangle$ of order $n$. Here, $\Si_n$ acts $k$-linearly on $k_n$ via $\Si_n(\E^\frac{j}{n})=\zeta_n^j\E^\frac{j}{n}$, with $\zeta_n\coloneqq  e^\frac{2\pi i}{n}$.
 
 For $m|n$ we identify $k_m$ with the subfield of $k_n$, which is generated by $(\E^\frac{1}{n})^\frac{n}{m}$. Thus, $\Si_n|_{k_m}=\Si_m$.
 
 We say that $x\in k_n$ is {\em generic} if $|C_n \cdot  x|=n$, or equivalently, $\Pi_{j=0}^{n-1}(y-\Si_n^j(x))\in k[y]$ is an irreducible polynomial. For $0\neq x=\sum_{j\in\Z}x_j\E^\frac{j}{n}\in k_n$, we define $$\cogrado_n(x)\coloneqq \min\{j\in\Z|x_j\neq0\}$$
 and  set $\cogrado_n(0):=\infty$.
 
 \subsection{Representations of the species of type \texorpdfstring{$(1,\,4)$}{(1,4)} over \texorpdfstring{$k$}{k}}\label{sec1.3}
 We abbreviate $K\coloneqq \C\parEc=k[\E^\frac{1}{4}]$. A representation $M$ of  the $_KK_k$ bimodule is of the form $M\coloneqq(K^n,k^m,\varphi_M)$, where $\varphi_M\colon K\otimes_{k}k^{m}\lra K^{n}$ is $K$-linear and  $m,\,n\in \N$, \cite[Section 1]{Ri90}. Let $N$ be the matrix of the map $\varphi_M$.
 
Let $\Rep_k(1,\,4)$ be a category, where the objects are the matrices with entries in $K$. For $N\in \M_{n\times m}(K)$ and $N'\in \M_{n'\times m'}(K)$, the morphisms from $N$ to $N'$ are  given by the $k$-vector space:
 $$\Hom_{(1,\,4)}(N,\,N')\coloneqq \{(f_2,\,f_1)\in\M_{n'\times n}(K)\times\M_{m'\times m}(k)|f_2N=N'f_1\}.$$
 
It is easy to see that $\Rep_k(1,\,4)$ is a $k$-linear abelian category. This category  coincides with the category of bimodule representations from \cite[Section 1]{Ri90}. It is naturally equivalent to $\Lambda\mbox{-}\modu$ for the Euclidean species $\Lambda=\left[\begin{array}{cc}K&K\\0&k\end{array}\right]$ of type $(1,\,4)$ over $k$. In what follows, we will implicitly identify these two categories.
Consider  $N\in\M_{m\times n}(K)$ 
 as a representation in $\Rep_k(1,\,4)$  in the
above sense, we  
 write $\underline{\dim}N=(m,\,n)\in\N_0^2$
 for its \emph{dimension vector}.
 
 Following Dlab-Ringel \cite[Section 3]{DR75}, $N\in\Rep_k(1,\,4)$ is {\em regular} if for each direct summand $N'$ of $N$ we have $\underline{\dim}N'=(m,2m)$ for some $m\in\Z_{>0}$; if it also satisfies that $\End_{(1,\,4)}(N):=\Hom_{(1,\,4)}(N,\,N)$ is a division algebra, i.e. isomorphic to $k_n$ for some $n$ (because $k$ is quasi-finite), we say that it is simple regular.
 The simple regular modules are the simple objects of the abelian category of regular representations of $\Lambda$, see \cite[Section 3]{DR75}.
 
 In this paper we will only consider the quiver represented in Figure \ref{figd4}, which, abusing notation, we will denote by $\widetilde{\mathsf{D}}_4$.
 \begin{figure}
  \centering
\begin{tikzpicture}
[->,>=stealth',shorten >=1pt,auto,node distance=2cm,thick,main node/.style=]
  \node[main node] (1) {$4$};   
  \node[main node] (2) [below right of=1] {$0$};
  \node[main node] (3) [below left of=2] {$3$};
  \node[main node] (4) [above right of=2] {$1$};
  \node[main node] (5) [below right of=2] {$2$};
\path[every node/.style={font=\sffamily\small}]     
(2) edge node   {$\alpha_4$} (1)         
    edge node         {$\alpha_1$} (4)             
  edge node  {$\alpha_3$} (3)              
   edge node         {$\alpha_2$} (5)
; 
\end{tikzpicture}
\caption{The quiver $\widetilde{\mathsf{D}}_4$.}
     \label{figd4}
 \end{figure}

We will see in Section \ref{isokd}  that, for $4|d$, we obtain $k_d\otimes_k \Lambda\cong k_d\widetilde{\mathsf{D}}_4$. The path algebra $k_d\widetilde{\mathsf{D}}_4$ admits a $k$-linear automorphism $\G_d$ of order $d$, which is compatible with the action of $\Si_d$ on $k_d\otimes\Lambda$, see Sections \ref{isokd} and  \ref{Invariantrep} for more details.

Moreover, let $M\in\Rep_k(1,\,4)$ be a simple regular representation and $\End_{(1,\,4)}(M)\cong k_n$. Then, with $d=\mcm(4,\,n)$, we have $k_d\otimes_kM=\bigoplus_{j=0}^{n-1}\,^{\G_d^j}M'$ for a simple regular $k_d$-representation $M'$ of $\widetilde{\mathsf{D}}_4$, with $\End_{k_d\widetilde{\mathsf{D}}_4}(M')\cong k_d$ and $\,^{\G_d^n}M'\cong M'$. % Following \cite{Hu02}, we say that  $M$  is {\em isomorphically invariant representation} (ii-representation) if  $\,^{\G_d}M\cong M$. We say that $M$ is an {\em ii-indecomposable} if it is not isomorphic to the proper direct sum of two or more ii-representations. 
Following \cite[Section 3]{Hu04}, we say that  $M$  is an {\em isomorphically invariant representation} (ii-representation) if  $\,^{\G_d}M\cong M$. We say that $M$ is an {\em ii-indecomposable} if it is not isomorphic to the proper direct sum of two ii-representations. As we will see in Section \ref{Invariantrep}, the $ii$-indecomposable representations of $\widetilde{\mathsf{D}}_4$ are not just parameterized by $\mathbb{P}^1(k)\setminus\{0,\,1,\,\frac{1}{2}\}$. Our classification is a bit more complicated.

\subsection{Notation}\label{mainresult}
For $n\in\Z_{>0}$, we define the matrix $\A_n\in\M_{n\times n}(k)$ by
$$(\A_n)_{ij}=\left\{\begin{array}{ccc}\delta_{i+1,\,j}&\mbox{if}&i<n,\\\delta_{1,\,j}\E&\mbox{if}&i=n,\end{array}\right.$$
i.e., 
$$\A_n=\left[\begin{array}{ccccc}0&1&0&\cdots&0\\ 0&0&1&\cdots&0\\\vdots&\vdots&\vdots&\ddots&\vdots\\0&0&0&\cdots&1\\\E&0&0&\cdots&0\end{array}\right],$$
and we agree that $\A_1=\E$.

It is easy to see that the minimal polynomial of $\A_n$ is $y^n-\E\in k[y]$, and that each element $x\in k_n$ can be written in a unique way as
$$x=\sum_{j=0}^{n-1}x_j\E^\frac{j}{n},$$
for certain $x_j\in k$.
 \begin{lema} \label{lema2} We have an isomorphism of fields $k_n\tilde{\rightarrow}k[\A_n]\subset\M_{n\times n}(k)$ which sends $x=\sum_{j=0}^{n-1}x_j\E^\frac{j}{n}$ to
$$\A(x)\coloneqq \sum_{j=0}^{n-1}x_j\A_n^j.$$
Moreover, there exists $A\in\GL_n(k_n)$ such that for each $x\in k_n$, we have 
$$A^{-1}\A(x)A=\diag(x,\,\Si_n(x),\cdots,\,\Si_n^{n-1}(x)).$$
  \end{lema}
  \begin{proof}Since the minimal polynomial of $\A_n$ is $x^n-\E$, it is easy to see that $k_n\cong k[\A_n]$ with the isomorphism given by $\E^\frac{1}{n}\mapsto\A_n$.
   We define the matrix $A=(a_{ij})$ as 
  $$A=\left[\begin{array}{cccc}1&1&\cdots&1\\\E^\frac{1}{n}&\zeta_n\E^\frac{1}{n}&\cdots&\zeta_n^{n-1}\E^\frac{1}{n}\\
  \vdots&\vdots&\ddots&\vdots\\\E^\frac{n-1}{n}&\zeta_n^{n-1}\E^\frac{n-1}{n}&\cdots&\zeta_n^{(n-1)^2}\E^\frac{n-1}{n}\end{array}\right]\in\M_{n\times n}( k_n)$$ 
 i.e. $a_{ij}=\zeta_n^{(j-1)(i-1)}\E^\frac{i-1}{n},\mbox{ for } 1\leq i,\,j\leq n$. Note that $A$ is the  transpose of a Vandermonde matrix. Hence, $\det(A)=\prod_{1\leq i<j\leq n}\E^\frac{1}{n}(\zeta_n^j-\zeta_n^i)\neq0$, and thus, $A\in\GL_n(k_n)$. It is easy to see that  $A^{-1}=(\frac{y_{ij}}{n\E})$, where $y_{ij}=\zeta_n^{(i-1)(1-j)}\E^\frac{n+1-j}{n}\mbox{ for } 1\leq i,\,j\leq n$. 
 
 Recalling that $(\A_n^l)_{ij}=\left\{\begin{array}{lcl}\D_{i+l,\,j}&\mbox{if}&1\leq i<n+1-l\\\D_{i+l-n,\,j}\E&\mbox{if}&n+1-l\leq i\leq n\end{array}\right.$, it follows that 
 $$(A^{-1}\A_n^l)_{ij}=\left\{\begin{array}{lcl}\frac{1}{n}y_{i\,n-l+j}&\mbox{if}&1\leq j\leq l\\
 &&\\ \frac{1}{n\E}y_{i,\,j-l}&\mbox{if}&l+1\leq j\leq n.\end{array}\right.$$
 Then, for $A^{-1}\A_n^lA$: 
  \begin{eqnarray*}
  (A^{-1}\A_n^lA)_{ij}&=&\frac{1}{n\E}\left(\sum_{r=1}^l\E y_{i\,n-l+r}x_{rj}+\sum_{r=l+1}^n y_{i\,r-l}x_{rj}\right)\\
&=&\frac{1}{n\E}\left(\E \sum_{r=1}^l\zeta_n^{(i-1)(l+1-r)}\E^\frac{l+1-r}{n}\zeta_n^{(r-1)(j-1)}\E^\frac{r-1}{n}\right.\\
&+&\left.\sum_{r=l+1}^n \zeta_n^{(i-1)(n+l+1-r)}\E^\frac{n+l+1-r}{n}\zeta_n^{(r-1)(j-1)}\E^\frac{r-1}{n}\right)\\
&=&\frac{1}{n\E}\left(\E \sum_{r=1}^l\zeta_n^{(i-1)l}\E^\frac{l}{n}\zeta_n^{(r-1)(j-i)}+\sum_{r=l+1}^n \zeta_n^{(i-1)l}\E^\frac{n+l}{n}\zeta_n^{(r-1)(j-i)}\right)\\
&=&\frac{\E^\frac{l}{n}\zeta_n^{(i-1)l}}{n}\left(\sum_{r=1}^n\zeta_n^{(r-1)(j-i)}\right)\\
 &=&\D_{ij}\E^\frac{l}{n}\zeta_n^{(i-1)l}.
\end{eqnarray*}
Therefore, since $A^{-1}\A_n^lA=\diag(\E^\frac{l}{n},\,\E^\frac{l}{n}\zeta_n^{l},\cdots,\,\E^\frac{l}{n}\zeta_n^{(n-1)l})$, it follows that 
$$A^{-1}x(\A_n)A=\diag(x,\,\Si_n(x),\cdots,\,\Si_n^{n-1}(x)).$$
\end{proof}

\begin{defi} \label{def:regular}
Using the notation from Lemma \ref{lema2}, each $R\in K[\A_n]\subset\M_{n\times n}(K)$ can be written in a unique way as
$$R=\A(\nu_0)+\A(\nu_1)\E^\frac{1}{4}+\A(\nu_2)\E^\frac{1}{2}+\A(\nu_3)\E^\frac{3}{4}$$ 
for certain $\nu_0,\,\nu_1,\,\nu_2,\,\nu_3\in k_n$. We say that $R\in K[\A_n]$ is {\em regular} if:
\begin{enumerate}
\item[(R1)] $\nu_1^2-\E\nu_3^2\neq0$.
\item[(R2)] $\Si_n^m(\nu_2)\neq - \frac{(1-i)\E^\frac{3}{4}}{2\E}(\Si_n^m(\nu_1)+i\Si_n^m(\nu_3)\E^\frac{1}{2})$ for $1\leq m\leq n$.%\footnote{ This implies that $R-\Si(R)$ and $R-\Si^2(R)$ are invertible matrices.}.
\end{enumerate}
Then, we define the set $K[\A_n]_{\reg}\coloneqq\{R\in K[\A_n]| R\mbox{ is regular}\}$ and $K[\A_n]_{R1}\coloneqq\{R\in K[\A_n]| R\mbox{ satisfies condition (R1)}\}$.
\end{defi}

%Note, that $K[\A_n]\cong\left\{\begin{array}{lcl}k_{4n}&\mbox{if}&n\equiv 1\,(\modu 2),\\k_{2n}\times k_{2n} &\mbox{if}&n\equiv 2\,(\modu 4),\\k_{n}\times k_{n}\times k_{n}\times k_{n} &\mbox{if}&n\equiv 0\,(\modu 4).\end{array}\right.$

%We will prove this in Section \ref{2.3}.

For each $R\in K[\A_n]$, we  define the representation $M(R)\coloneqq(K^n,\,k^{2n}, \varphi_{M(R)})$, where $\varphi_{M(R)}$ is a $K$-linear map  and the matrix of $\varphi_{M(R)}$ is $[\id_n|R]$. For $X\in k[\A_n]$ be an invertible matrix, we have $M(XR)$ is isomorphic to $M(R)$.

 Finally we introduce the map  
 $\phi_n\colon K[\A_n]_{R1}\rightarrow k_n$,
 which is given by 
$$\A(\nu_0)+\A(\nu_1)\E^\frac{1}{4}+\A(\nu_2)\E^\frac{1}{2}+\A(\nu_3)\E^\frac{3}{4}\mapsto -i\frac{\nu_2^2-\nu_1\nu_3}{\nu_1^2-\nu_3^2\E}.$$
The map was suitably chosen to allow us to prove later on Proposition \ref{prop2.8}.

\subsection{\texorpdfstring{$\widetilde{\mathsf{D}}_4$}{tD4}-homogeneous representations
of \texorpdfstring{$\Lambda$}{Λ}}
Let us introduce for each $n\in\Z_{>0}$
the following two subsets of $k_n$:
$$I_n^{(1)}=\left\{\begin{array}{lcl}k_n\setminus\{\E^\frac{1}{2}\}&\mbox{if}&n\equiv0(\modu 4)\\
k_n&\mbox{if}&n\not\equiv0(\modu 4)\end{array}\right.\mbox{ and } I_n^{(2)}=\left\{\begin{array}{lcl}k_n\setminus\{0\}&\mbox{if}&n\equiv1(\modu 2)\\
\{\E^{-\frac{n-2}{4n}}\C[\![\E^\frac{1}{n}]\!]\}&\mbox{if}&n\equiv2(\modu 4).\end{array}\right.$$

We define $M_n^{(j)}(s)$ as the representation of $\Lambda$ with $M(R_n^{(j)}(s))$ as the defining matrix, where $s\in k_n$, $j\in \{1,\,2\}$ and $R_n^{(j)}(s)$ is a matrix in $K[\A_n]$, see Table \ref{table1}. It is worth mentioning that for the case where $n$ is congruent to 0 modulo 4, we only consider $I_n^{(1)}$.

Now, we can define $\phi_n^{(j)}$ as a map from $I_n^{(j)}$ to $k_n$, which is defined as $s\mapsto\phi_n(R_n^{(j)}(s))$.% for $[\id_n|R_n^{(j)}(s)]$.

An interesting fact is that for any $t\in\C\lp x\rp$, $\sqrt{t}$ belongs to $\C\lp x^\frac{1}{2}\rp$. Specifically, the following lemma describes the form of a square root for any given element t. 
\begin{lema}\label{cuadrado}
 Let $\C\lp x\rp$ be the field of complex Laurent series. Then, for every $t\in\C\lp x\rp$, there exists $s\in\C\lp x\rp$ such that
$$t=\left\{\begin{array}{ll}s^2&\mbox{if  $\cogrado(t)$ even,}\\xs^2&\mbox{if $\cogrado(t)$ odd.}\end{array}\right.$$
\end{lema}
\begin{proof}
Let $t=\sum_{j=n_0}^\infty t_jx^j$, where $t_j\in\C$ and $n_0=\cogrado(t)$. On the one hand, for $n_0$ even and $m=n_0/2$, it is easy to see that $s^2=t$, where $s=\sum_{l=m}^\infty s_lx^l\in\C\lp x\rp$, and  $s_l$  is defined recursively with respect to the previous elements as follows:
$$s_{m+j}=\left\{\begin{array}{ll}\sqrt{t_{n_0}}&j=0\\\frac{t_{n_0+j}-\sum_{k=1}^{j-1}s_{m+k}s_{j+m-k}}{2s_{m}}&\mbox{otherwise.}\end{array}\right.$$
On the other hand, for $n_0$ is odd, we have $t=x\sum_{j=n_0-1}^\infty t_{j+1}x^j$. In this case, we can find the corresponding square root of $\sum_{j=n_0-1}^\infty t_{j+1}x^j$. 
\end{proof}
\begin{prop}\label{prop1.4}
Let $n$ be a positive integer. The following statements hold:
\begin{enumerate}
\item $k_n=\ima(\phi_n^{(1)})\dot\cup \,\ima(\phi_n^{(2)})$ for $n\not\equiv 0(\modu 4)$.
\item $k_n\setminus\{0\}=\ima(\phi_n^{(2)})$ for $n\equiv 0(\modu 4)$.
\end{enumerate}
\end{prop}
\begin{proof}
\begin{enumerate}
\item We have two cases: $n\equiv1(\modu2)$ and $n\equiv2(\modu4)$. On the one hand, for $n\equiv1(\modu2)$, the result follows from the Lemma \ref{cuadrado}. On the other hand, for $n\equiv2(\modu4)$ and $t\in\ima(\phi_n^{(2)})$, we have $1-4t^2=(\E^\frac{1}{n}s)^2(\E^\frac{1}{2}+\E^\frac{2+n}{n}s^2)$ for some $s\in I_n^{(2)}$. Thus, $\cogrado(s)\geq -\frac{n-2}{4}$. It  follows that $\cogrado(\E^\frac{1}{2})<\cogrado(\E^\frac{2+n}{n}s^2)$. Then, $\cogrado(1-4t^2)$ is odd, and therefore $\sqrt{1-4t^2\E}\notin k_n$. It follows that $\ima(\phi_n^{(1)})$ and $\ima(\phi_n^{(2)})$ are disjoint. Let $t\in k_n$ such that $t\notin\ima(\phi_n^{(2)})$. Then, 
$$s\coloneqq i\frac{1+\sqrt{1-4t^2\E}}{2t}\in k_n$$
and $s$ satisfies $\frac{is}{s^2-\E}=t$. So $k_n=\ima(\phi_n^{(1)})\dot\cup\ima(\phi_n^{(2)})$.
\item Suppose $n\equiv0(\modu4)$. For $t\in k_n\setminus\{0\}$, we can choose $s=\frac{-i+t\E^\frac{1}{2}}{t}\in k_n\setminus\{\E^\frac{1}{2}\}$. Then, $\frac{-i}{s-\E^\frac{1}{2}}=t$, and therefore, $k_n\setminus\{0\}=\ima(\phi_n^{(2)})$.
\end{enumerate}
\end{proof}

\begin{table}[ht]
\centering
\begin{tabular}{|c|c|c|c|}\hline
Parity of $n$&j  & $R_n^{(j)}(s)$ & $\phi_n^{(j)}(s)$\\ \hline
&&&\\
$n\equiv1\,(\modu 2)$&$1$  &$\A(s)\A_n^\frac{n+1}{2}\E^\frac{1}{2}+e^{\frac{1}{4}\pi i}\id_n\E^\frac{3}{4} $& $\E^\frac{1}{n}s^2$\\ &&&\\ \hline &&&\\
$n\equiv1\,(\modu 2)$&$2$  &$e^{\frac{3}{4}\pi i}\id_n\E^\frac{1}{4}+\A(s)\E^\frac{1}{2}$ &$s^2$\\ &&&\\ \hline &&&\\
$n\equiv2\,(\modu 4)$&$1$  &$\A(s)\E^\frac{1}{4}+\id_n\E^\frac{3}{4} $&$i\frac{s}{s^2-\E}$\\ &&&\\ \hline &&&\\
$n\equiv2\,(\modu 4)$&$2$  &$i\A_n^\frac{n}{2}\E^\frac{1}{4}+(1-i)\A(s)\A_n^{\frac{n}{2}+1}\E^\frac{1}{2}+\id_n\E^\frac{3}{4} $& $\frac{1}{2}\E^\frac{-1}{2}+(\E^\frac{1}{n}s)^2$\\ &&&\\ \hline
 &&&\\
$n\equiv0\,(\modu 4)$&$1$  &$\A(s)\E^\frac{1}{4}+\A_n^{\frac{n}{4}}\E^\frac{1}{2}-\id_n\E^\frac{3}{4} $&$ \frac{-i}{s-\E^\frac{1}{2}}$\\ &&&\\ \hline
\end{tabular}
\caption{Ingredients for the description of the homogeneous simple regular representations.}
\label{table1}
\end{table}

We define the representations $M_{(1)}\coloneqq(K,\,k^{2}, \varphi_{M_{(1)}})$ and $M_{(2)}\coloneqq(K^2,\,k^{4}, \varphi_{M_{(2)}})$, where $[1,\,\E^\frac{1}{2}]$ and $[\id_2|\A_2\E^\frac{1}{4}-i\id_2\E^\frac{3}{4}]$ are the matrices of the $K$-linear maps $\varphi_{M_{(1)}}$ and $\varphi_{M_{(2)}}$, respectively. As we will see in Theorem \ref{lema1}, these representations are not isomorphic to any of those in Table \ref{table1}.

\begin{defi}\label{defhomo}
    The simple regular representations of $\Lambda$, which are not isomorphic to $M_{(1)}$ or $M_{(2)}$, are called  $\widetilde{\mathsf{D}}_4$-homogeneous.
\end{defi}

Now, according to the definition of $\widetilde{\mathsf{D}}_4$-homogeneous, we can derive the following result as a consequence of Theorem~\ref{lema1} below and
Proposition~\ref{prop1.4}.

\begin{teo}\label{teorema1.5} Let $M$ be a $\widetilde{\mathsf{D}}_4$-homogeneous simple regular representation of $\Lambda$ with $\underline{\dim}M=(n,\,2n)$. Then, $M$ is isomorphic to a representation of the form $M_n^{(j)}(s)$ with $s\in I_n^{(j)}$ for some $j\in\{1,\,2\}$.
\end{teo}

\begin{remark}  These are  ``weak'' normal forms of the simple regular representations of $\Lambda$, which are $\widetilde{\mathsf{D}}_4$-homogeneous in the sense that  there are at most $2n$ different elements in $I_n^{(j)}$ that yield the ``same" simple regular representation. Our Theorem also  implies that the $\widetilde{\mathsf{D}}_4$-homogeneous simple regular representations are parameterized by $\spec(k[x])$. Here,  the relation is that  every non-zero prime ideal in $\spec(k[x])$ is generated by a unique monic irreducible polynomial $p\in k[x]$. Consider  a root $\zeta\in k_n$   of $p$, then by Proposition \ref{prop1.4}  and  Theorem \ref{teorema1.5}, we can find the $s_\zeta \in I_n^{(j)}$, for some $j\in\{1,\,2\}$, such that $M_n^{(j)}(s_\zeta)$ is a simple regular representation.
\end{remark}

%%%%%%%%%%%%%%%%%%%%%%%%%%%%%%%%%
\section{Twisted quiver representations and Galois descent}\label{sec3}
\subsection{Galois descent}
In this section we will discuss the technique of Galois descent in more detail. Let $k_d/k$ be a finite Galois extension with Galois group $\langle\sigma\rangle$. We will explore the conditions that a representation $M$ of $k_d\widetilde{\mathsf{D}}_4$ must satisfy to find an $X$ a $\Lambda$ module such that $k_d\otimes_k X\cong M$.
For unexplained notation and definitions we refer to \cite{DR75,GLS16,RR85,Ri90}.

We will use the following classical facts:
\begin{enumerate}
\item Let $n,\,m\in \Z_{>0}$ and $k_n$ as in \ref{quasi1.2}. Then, $k_n\otimes_k k_m\cong \underbrace{k_l\times\cdots\times k_l}_{g\mbox{-times}}$, as $k_n$-algebras, with $l=\mcm(n,\,m)$ and $g=\mcd(n,\,m)$.
\begin{proof}
Let $m,\,n\in\Z_{>0}$, $l=\mcm(n,\,m)$ and $g=\mcd(n,\,m)$.  Define 
$$\B_j=\frac{1}{g\E}\sum_{r=0}^{g-1}\E^\frac{g-r}{g}\otimes_k\Si_g^{j-1}(\E^\frac{r}{g})\in k_n\otimes_k k_m$$ for $1\leq j\leq g$. It can be checked by elementary computation that the $\B_j$'s satisfy $\B_j\B_i=\D_{ji}\B_j$. Since $\mcd(\frac{l}{n},\,\frac{l}{m})=1$, we have from B\'ezout's identity that there are $a,\,b\in\Z$ such that $\left(\frac{l}{n}\right)a+\left(\frac{l}{m}\right)b=1$. It is easy to see that 
$$\left\{\E^\frac{au}{n}\B_j\E^\frac{bu}{m}|1\leq j\leq g,\,0\leq u\leq \frac{l}{n}-1\right\}$$
is a basis of $k_n\otimes_kk_m$ as a $k_n$-vector space.

Thus, we obtain the $k_n$-algebra isomorphism  $$k_n\otimes_k k_m\cong \underbrace{k_l\times\cdots\times k_l}_{g\mbox{-times}},$$
where $\E^\frac{au}{n}\B_j\E^\frac{bu}{m}\mapsto \E^\frac{u}{l}e_j$, with $\{e_j|1\leq j\leq g\}$ being the standard basis of $\underbrace{k_l\times\cdots\times k_l}_{g\mbox{-times}}$.
\end{proof}
\item Let $L/k$ be a Galois extension, with Galois group $C_d=\langle\Si\rangle$ and $d$ the degree of the extension. Then, the skew group algebra (see \cite{RR85}) $L*C_d$ is isomorphic to the matrix ring $\M_{d\times d}(k)$ as a $k$-algebra. This seems to go back to, at least, Auslander and Goldman \cite[Appendix]{AG60}. See also \cite[Appendix A64]{Mi17}.
\item Let $A$ be a $k$-algebra on which $C_d=\langle\Si\rangle$ acts by $k$-linear automorphisms. Then, the category of $A*C_d$-modules is equivalent to the category of pairs $(M,\,f)$, where $M\in A$-$\modu$ and $f\in \Hom_A(\,^\Si M,\,M)$ such that $f\circ\,^\Si f\circ\cdots \circ\,^{\Si^{d-1}}f=\id_M$. See \cite{RR85} or \cite[Section 3.2]{Hu02} for more details.
\item For $\Lambda=\left[\begin{array}{cc}K&K\\0&k\end{array}\right]$  and $4|d$, we have $\psi_d\colon k_d\otimes_k\Lambda\stackrel{\sim}{\lra}k_d\widetilde{\mathsf{D}}_4$ as $k_d$-algebras, and we describe in Section \ref{isokd} how $\langle\Si\rangle$ acts under $\psi_d$ on $k_d\widetilde{\mathsf{D}}_4$ and $k_d\widetilde{\mathsf{D}}_4$-$\modu$. We have $\Si$ is compatible with $\psi_d$. Then, $(k_d\widetilde{\mathsf{D}}_4)*\langle\Si\rangle\cong(k_d\otimes_k\Lambda)*\langle\Si\rangle$. It follows from  b) that $(k_d\otimes_k\Lambda)*\langle\Si\rangle\cong\M_{d\times d}(\Lambda)$ as $k$-algebras.  Therefore $(k_d\widetilde{\mathsf{D}}_4)*\langle\Si\rangle\cong\M_{d\times d}(\Lambda)$ as $k$-algebras.
In particular, $(k_d\widetilde{\mathsf{D}}_4)*\langle\Si\rangle$-$\modu$ is equivalent to $\Lambda$-$\modu$.
Moreover, for $X\in\Lambda$-$\modu$
$$\Si^{-1}\otimes_k\id_X\colon  k_d\otimes_k X\stackrel{\sim}{\lra}\,^\Si(k_d\otimes_k X)=\,^\Si k_d\otimes_k X$$
is a $k_d\otimes_k\Lambda$-module isomorphism.
\item Let $X\in\Lambda$-$\modu$ be simple regular. We have $\End_\Lambda(X)=k_m$ for some $m\in\N$, since $k=\C\parE$ is quasi-finite. We have by a):
\begin{eqnarray}
\End_{k_d\otimes_k\Lambda}(k_d\otimes_k X)\cong k_d\otimes_k\End_\Lambda(X)\cong \underbrace{k_d\times\cdots\times k_d.}_{m},
\end{eqnarray}
with $d=\mcm(4,\,m)$. Furthermore, $\tau_{k_d\widetilde{\mathsf{D}}_4}(k_d\otimes_k X)\cong k_d\otimes_k\tau_\Lambda(X)\cong k_d\otimes_k X$, where $\tau$ denotes the Auslander-Reiten translation. So, $k_d\otimes_k X\cong N_0\oplus\cdots\oplus N_{m-1}$ for simple regular modules $N_i$ with $\Hom_{k_d\widetilde{\mathsf{D}}_4}(N_i,\,N_j)\cong\D_{i,\,j}\,k_d$. 

Furthermore, since $^\Si(k_d\otimes_k X)\stackrel{\sim}{\lra}k_d\otimes_k X$, we can assume that
\begin{eqnarray}
\,^\Si N_i\cong N_{i+1}\mbox{ for }i=0,\,1,\cdots,\,m-2\mbox{ and } \,^\Si N_{m-1}\cong N_0.
\end{eqnarray}
It is easy to classify these families of $k_d\widetilde{\mathsf{D}}_4$-modules, see Section \ref{Invariantrep}.
\item Let $M=\bigoplus_{j=0}^{m-1}\,^{\Si^j}N$ for a simple regular $k_d\widetilde{\mathsf{D}}_4$-module $N$ with $\End_{k_d\widetilde{\mathsf{D}}_4}(N)\cong k_d$, $\Hom_{k_d\widetilde{\mathsf{D}}_4}(\,^{\Si^j}N,\,N)=0$ for $j=1,\,2,\cdots,\,m-1$, and $\,^{\Si^m}N\cong N$. It is easy to see that finding an isomorphism $f\colon \,^\Si M\lra M$ with $f\circ\,^\Si f\circ\cdots\circ\,^{\Si^{d-1}} f=\id_M$ is equivalent to finding $\tilde{f}\colon \,^{\Si^m}N\lra  N$, such that
\begin{eqnarray}\label{3asterisco}
\tilde{f}\circ\,^{\Si^m} \tilde{f}\circ\cdots\circ\,^{\Si^{km}} \tilde{f}=\id_N\mbox{ for } k\coloneqq \frac{d}{m}-1.
\end{eqnarray}
In fact, up to rescaling, $f$ must be of the form
$$f=\left[\begin{array}{ccccc}0&0&\cdots&0&\tilde{f}\\ \id_2&0&\cdots&0&0\\0&\id_3&\ddots&\vdots&0\\\vdots&\vdots&\ddots&\ddots&\vdots\\0&0&\cdots&\id_{m-1}&0\end{array}\right],$$
since $\Hom_{k_d\widetilde{\mathsf{D}}_4}(\,^{\Si^i}N,\,^{\Si^j}N)\cong\D_{i,\,j}\,k_d$, and thus,  
$$f\circ\,^\Si f\circ\cdots\circ\,^{\Si^{m-1}} f=\diag(\,^{\Si^i}f)_{i=0,\,1,\cdots,\,m-1}$$ and $\id_i=\id_{\,^{\Si^i}M}$.

We will exhibit  in Section \ref{Invariantrep} such an $\tilde{f}\colon \,^{\Si^m}N\lra N$ with property (\ref{3asterisco}) for each regular $k_d\widetilde{\mathsf{D}}_4$-module with $\End_{k_d\widetilde{\mathsf{D}}_4}(N)\cong k_d$ and minimal $m$ such that $\,^{\Si^m}N\cong N$ and $\frac{d}{m}\in\{1,\,2,\,4\}$.
\begin{remark}{
If $f'\colon \,^\Si M\lra M$ is another isomorphism with $f'\circ\,^\Si f'\circ\cdots\circ\,^{\Si^{d-1}} f'$ equal to $\id_M$, then again, since $\End_{k_d\widetilde{\mathsf{D}}_4}(N)\cong k_d$, we must (using the above notation) have $\tilde{f}'=c\tilde{f}$ for some $c\in k_d$ with 
$$N_{d,\,m}(c)\coloneqq c\Si^m(c)\cdots\Si^{km}(c)=1_{k_d}\,\,(k=\frac{d}{m}-1)$$
However, by the Lemma below, in this case, we can find $x\in k_d^*$ with $c=\Si^m(x)^{-1}\cdot x$. This implies that $(M,\,f)\cong (M,\,f')$ since we have 
\begin{center}
\begin{tikzpicture}
[->,>=stealth',shorten >=1pt,auto,node distance=2cm,thick,main node/.style=]
  \node[main node] (1) {$\,^{\Si^m}N$};   
  \node[main node] (2) [right of=1] {$N$};
  \node[main node] (3) [below of=1] {$\,{\Si^m}N$};
  \node[main node] (4) [below of=2] {$N$};
 \path[every node/.style={font=\sffamily\small}]     
(1) edge node   {$\tilde{f}$} (2)          
(1) edge node    {$\,^{\Si^m}(x\cdot\id_N)$} (3)           
(3) edge node  {$c\tilde{f}$} (4)
(2) edge node  {$x\cdot\id_N$}  (4)              
;
\end{tikzpicture}
\end{center}
}\end{remark}
The following lemma seems to be well-known. We include a proof for the convenience of the reader.
\begin{lema}If $c\in k_n$ with $N_{n,\,1}(c)=1$, then there exists $x\in k_n$ with $\Si(x)^{-1}x=c$.
\end{lema}
\begin{proof}
Recall that $k_n*C_n\stackrel{\sim}{\lra}\M_{n\times n}(k)$ is a simple algebra. Thus, $(k_n,\, \Si)$ corresponds to the unique (simple) $k_n*C_n$-module that has $k$-dimension $n$ up to isomorphism.

On the other hand, we have a homomorphism of abelian groups
$$\psi_n\colon k_n^*\lra\widetilde{N}_n\coloneqq\{c\in k_n^*|N_{m,\,1}(c)=1\}\subset k_n^*$$
$$x\mapsto \Si(x)^{-1}x,$$
and $\widetilde{N}_n/\ima\psi_n$ parameterizes  the isoclasses  of $k_n*C_n$-modules of the form $(k_n,\,c\Si)$ with $c\in \widetilde{N}_n$.

Thus, $\psi_n$ must be surjective, since there is only one isoclass of $n$-$\dim$ $k_n*C_n$-modules.
\end{proof}
\item Considering c) we conclude from f) that the isoclasses of simple regular $\Lambda$-modules with $m$-dimensional endomorphism ring ($\cong k_m$) are, for $d=\mcm(4, \,m)$, in bijection with the $C_d$-orbits of simple regular $k_d\widetilde{\mathsf{D}}_4$-modules $N$ with:
\begin{itemize}
\item $\End_{k_d\widetilde{\mathsf{D}}_4}(N)\cong k_d$,
\item $\,^{\Si^m}N\cong N$ and
\item $\,^{\Si^i}N\not\cong N$ for $i=1,\,2,\cdots,\,m-1$.
\end{itemize}
In Section \ref{Invariantrep}, we will find a simple regular $\Lambda$ module $X$ with $\End_\Lambda(X)=k_m$ and $k_m\otimes_k X\cong\bigoplus_{j=0}^{m-1}\,^{\Si^m}N$, for each $N$ with this property.
\end{enumerate}

  \subsection{Quivers with automorphisms}\label{isokd}

  In this section we assume that $d=\mcm(4,n)$. Let $\mathcal{B}=\{\B_1,\,\B_2,\,\B_3,\,\B_4\}$ be a basis of $k_d\otimes_k K$, where $\B_j$ is defined by 
 $$\frac{1}{4\E}\sum_{l=0}^3\E^\frac{4-l}{4}\otimes_k\Si^{j-1}(\E^\frac{l}{4}),$$ for $1\leq j\leq 4$. We obtain the $k_d$-algebra isomorphism  $k_d\otimes_k K\cong k_d\times k_d\times k_d\times k_d$. In fact, for the isomorphism, we only need $d$ be a multiple of $4$.

Recall that $\Lambda=\left[\begin{array}{cc} K& K\\0& k\end{array}\right]$. We can adapt the above basis to obtain the following $k_d$-basis of $k_d\otimes_{ k}\Lambda$:
 \begin{eqnarray}\label{base}
 \left\{\B_{jk}|1\leq j\leq4,\,\,1\leq k\leq 2,\,\B_{jk}=\frac{1}{4\E}\sum_{l=0}^3\E^\frac{4-l}{4}\otimes_k\Si^{j-1}(\E^\frac{l}{4})E_{1k}\right\}\cup\left\{\B_{01}=1\otimes_k E_{22}\right\}
 \end{eqnarray}
 where $\{E_{11},\,E_{12},\,E_{21},\,E_{22}\}$ is the standard basis of $\M_{2\times 2}(k_d)$. Then:
 \begin{eqnarray*}
 k_d\otimes_k\Lambda&\lra&  k_d \widetilde{\mathsf{D}}_4\mbox{, defined by}\\
 \B_{01}&\mapsto&\Ep_{0}\\
 \B_{j1}&\mapsto&\Ep_{j}\\
 \B_{j2}&\mapsto&\A_{j}
 \end{eqnarray*}
 is a $k_d$-algebra isomorphism, where $ k_d \widetilde{\mathsf{D}}_4$ is the path  algebra of $\widetilde{\mathsf{D}}_4\mbox{, }\Ep_j\in (\widetilde{\mathsf{D}}_4)_0$ and $\A_j\in (\widetilde{\mathsf{D}}_4)_1$. 
  
  On the other hand, take the $k$-automorphism $\Si_d\otimes_k\id\colon  k_d\otimes_k\Lambda\lra k_d\otimes_k\Lambda$, which acts on the basis in \eqref{base} as follows:
  \begin{eqnarray*}
  (\Si_d\otimes_k\id)(\B_{jk})&=&\B_{j-1,\,k}, \,1\leq k\leq 2\mbox{ and }(\Si_d\otimes_k\id)(\B_{01})=\B_{01}.
\end{eqnarray*}
Here we take  $j$ modulo $4$. Thus, we obtain the morphism 
$$\G_d\colon  k_d \widetilde{\mathsf{D}}_4\lra k_d \widetilde{\mathsf{D}}_4$$
$$k\rho\mapsto\Si_d(k)\G(\rho),$$
where $\G$ is the following automorphism of order 4 of the quiver $\widetilde{\mathsf{D}}_4$:

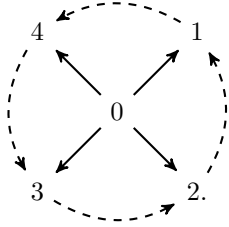
\begin{center}
\begin{tikzpicture}
[->,>=stealth',shorten >=1pt,auto,node distance=1.5cm,thick,main node/.style=]
  \node[main node] (1) {$4$};   
  \node[main node] (2) [below right of=1] {$0$};
  \node[main node] (3) [below left of=2] {$3$};
  \node[main node] (4) [above right of=2] {$1$};
  \node[main node] (5) [below right of=2] {$2$.};
\path[every node/.style={font=\sffamily\small}]     
(2) edge node   {} (1)         
    edge node     {}(4)             
  edge node   {}(3)              
   edge node  {}(5)
; 
\draw[->, dashed] (1) to[bend right]  (3);
\draw[->, dashed] (3) to[bend right]  (5);
\draw[->, dashed] (4) to[bend right]  (1);
\draw[->, dashed] (5) to[bend right]  (4);
\end{tikzpicture}
\end{center}

\subsection{Invariant representations}\label{Invariantrep}
Consider a simple regular representation $M$, i.e. $\varphi_M\colon K\otimes_k k^{2n}\lra K^{n}$ and  $\End_\Lambda(M)=k_m$, for some $m\in\Z_{>0}$. Let $N$ be the matrix of $\varphi_M$. For $d=\mcm\{4,n\}$, apply $k_d\otimes_{k}-$ to our  representation. We have $\End_{k_d\otimes_k\Lambda}(k_d\otimes_k M)\cong k_d\otimes_k \End_{\Lambda}(M)\cong k_d\times\cdots\times k_d$. Then, $k_d\otimes_k M$ can be identified with the following representation:

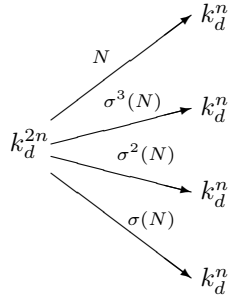
\begin{center}
 \begin{picture}(100,110)
  \put(85,2){\makebox(0,0){$  k_d^n$}}
\put(85,35){\makebox(0,0){$  k_d^n$}}
\put(85,67){\makebox(0,0){$  k_d^n$}}
\put(85,102){\makebox(0,0){$  k_d^n$}}
  \put(15,54){\makebox(0,0){$  k_d^{2n}$}}
  \put(42,87){\makebox(0,0){$\scriptstyle N$}}
  \put(54,70){\makebox(0,0){$\scriptstyle\Si^3(N)$}}
  \put(59,50){\makebox(0,0){$\scriptstyle\Si^2(N)$}}
  \put(61,25){\makebox(0,0){$\scriptstyle\Si(N)$}}
  \put(23,63){\vector(4,3){53}}
  \put(23,54){\vector(4,1){53}}
  \put(23,49){\vector(4,-1){53}}
  \put(23,43){\vector(4,-3){53}}
\end{picture}
\end{center}
where $\Si^j(N)$ is, by definition, the matrix where $\Si^j$ is applied to each component of $N$, and $\Si\coloneqq \Si_4$. This is a $k_d$-linear representation of the quiver $\tilde{\mathsf{D}}_4$.

The following is adapted from \cite{Hu02}. Let $\calm$ be a $ k_d \widetilde{\mathsf{D}}_4$-module. We define a module $\,^{\G_d}\calm$ by taking the same underlying $k$-vector space  as $\calm$, but with a new action
$$p*m=\G_d^{-1}(p)m\mbox{, for }p\in k_d \widetilde{\mathsf{D}}_4.$$
If $f\colon \calm\lra\caln$ is a module homomorphism, we obtain $\,^{\G_d}f\colon \,^{\G_d}\calm\lra\,^{\G_d}\caln$ a homomorphism as follows. As a linear map, we set $\,^{\G_d}f=f$. Then,
$$f(p*m)=f(\G_d^{-1}(p)m)=\G_d^{-1}(p)f(m)=p*f(m).$$
This yields a $k_d$-linear autoequivalence $F(\G_d)$ of $\modu k_d\widetilde{\mathsf{D}}_4$ such that $F(\G_d^r)=F(\G_d)^r$. Note that, in particular $\calm$ is indecomposable if and only if $\,^{\G_d}\calm$ is indecomposable as well.

We know that the categories $k_d \widetilde{\mathsf{D}}_4$-$\modu$  and $\Rep (\widetilde{\mathsf{D}}_4, k_d)$ are equivalent, so the functor $F(\G_d)$ must act on $\Rep( \widetilde{\mathsf{D}}_4, k_d)$. Let $M=(V_j,f_{\A_j})$ be a  $ k_d$-representation of $\widetilde{\mathsf{D}}_4$  and $\calm$ the corresponding  $ k_d \widetilde{\mathsf{D}}_4$-module, so that $\calm$ has underlying $k$-vector space $V=\bigoplus_{j=0}^4 V_j$. We want to describe the representation $\,^{\G_d} M=(W_j,g_{\A_j})$ corresponding to the module $\,^{\G_d}\calm$ in terms of the original representation $M$. Obviously,
$$W_j=\Ep_j*V=\G^{-1}(\Ep_j)V= \Ep_rV=V_r,$$ 
with $r\equiv j-1(\modu 4)$ for $1\leq j\leq 4$, while for $j=0$ we have
$$W_0=\Ep_0*V=\G_d^{-1}(\Ep_0)V= \Ep_0V=V_0.$$
Note that the $k_d$-vector space structure on $W_j$ is defined by the product $k*v_j=\Si_d^{-1}(k)v_j$, for all $k\in k_d$ and $v_j\in W_j$. On the other hand, for each $f_{\A_j}\colon V_0\lra V_j$, we can associate a matrix  $A_j\in\M_{\dim V_j\times\dim V_0}( k_d)$. Then, $g_{\A_j}\colon W_0\lra W_j$ is defined by  
$$g_{\A_j}(v)=\A_j*v=\G_d^{-1}(\A_j)v=\A_rv=A_rv=\Si_d^{-1}(\Si_d(A_r))v=\Si_d(A_r)*v,$$
where $\Si_d(A_r)$ means that $\Si_d$ is applied to each entry of the matrix $A_r$. So, for a  $k_d$-representation $M$ of $\widetilde{\mathsf{D}}_4$ with the following form
\begin{center}
 \begin{picture}(10,85)
  \put(0,75){\makebox(0,0){$M\coloneqq $}}
   \end{picture}  \begin{tikzpicture}
[->,>=stealth',shorten >=1pt,auto,node distance=1.5cm,thick,main node/.style=]
\node (0) at (0,2.25) {$V_0$}; 
\node (4) at (2.5,0) {$V_4$,}; 
\node (3) at (2.5,1.5) {$V_3$}; 
\node (2) at (2.5,3) {$V_2$}; 
\node (1) at (2.5,4.5) {$V_1$}; 
\node (5) at (1.5,3.1) {$A_2$}; 
\path[every node/.style={font=\sffamily\small}]     
(0) edge node    {$A_1$} (1)         
     edge node    {$A_4$} (4)             
     edge node    {$A_3$} (3)              
     edge node    {} (2);
 \end{tikzpicture}
\end{center}
the  representation $\,^{\G_d} M$ has the form:
\begin{center}
 \begin{picture}(14,85)
  \put(0,73){\makebox(0,0){$\,^{\G_d}M=$}}
\end{picture} \begin{tikzpicture}
[->,>=stealth',shorten >=1pt,auto,node distance=1.5cm,thick,main node/.style=]
\node (0) at (0,2.25) {$V_0$}; 
\node (4) at (2.5,0) {$V_1$.}; 
\node (3) at (2.5,1.5) {$V_4$}; 
\node (2) at (2.5,3) {$V_3$}; 
\node (1) at (2.5,4.5) {$V_2$}; 
\node  at (1.5,3.1) {$\scriptstyle\Si_d(A_3)$}; 
\node  at (1.72,1.2) {$\scriptstyle\Si_d(A_1)$}; 
\path[every node/.style={font=\sffamily\small}]     
(0) edge node    {$\scriptstyle\Si_d(A_2)$} (1)         
     edge node    {} (4)             
     edge node    {$\scriptstyle\Si_d(A_4)$} (3)              
     edge node    {} (2);
 \end{tikzpicture}
\end{center}

Following \cite[Section 3]{Hu04}, we say that  $M$  is an {\em isomorphically invariant representation} (ii-representation) if  $\,^{\G_d}M\cong M$. We say that $M$ is an {\em ii-indecomposable} if it is not isomorphic to the proper direct sum of two ii-representations. As in \cite{Hu02} {Lemma 2.3.1}, the ii-indecomposables are precisely the representations $M$ of the form 
$$M\cong N\oplus\,^{\G_d}N\oplus\,^{\G_d^2}N\oplus\cdots\oplus\,^{\G_d^{n-1}}N,$$
where $N$ is an indecomposable representation, and $n\geq 1$ is minimal such that $N\cong \,^{\G_d^{n}}N$.
Also, we say that $M$ is an {\em ii-simple regular} if $M$ is ii-indecomposable and  all summands are simple regular (see Section \ref{sec1.3}).

 Now we introduce the representations $M_t$ which, as we will see later, are the direct summands of the ii-indecomposable representations for certain $t$.  For $n\in\N$, $d=\mcm(n,4)$ and $t\in k_d$, the representation
 \begin{center}
\begin{picture}(60,101)
  \put(73,0){\makebox(0,0){$ k_d$,}}
\put(73,34){\makebox(0,0){$ k_d$}}
\put(73,67){\makebox(0,0){$ k_d$}}
\put(73,101){\makebox(0,0){$ k_d$}}
\put(-23,53){\makebox(0,0){$M_t\coloneqq $}}
  \put(0,53){\makebox(0,0){$ k_d^2$}}
  \put(20,88){\makebox(0,0){$\scriptstyle(1\,0)$}}
  \put(33,66){\makebox(0,0){$\scriptstyle(0\,1)$}}
  \put(43,49){\makebox(0,0){$\scriptstyle(1\,1)$}}
  \put(51,24){\makebox(0,0){$\scriptstyle(1\,t)$}}
  \put(10,62){\vector(4,3){53}}
  \put(10,53){\vector(4,1){53}}
  \put(10,48){\vector(4,-1){53}}
  \put(10,42){\vector(4,-3){53}}
\end{picture}
\end{center}
has $\End_{k_d\widetilde{\mathsf{D}}_4} M_t\cong k_d$. Thus, it is indecomposable and regular. Moreover, for $t,s\in k_d$, we have $\Hom_{k_d\widetilde{\mathsf{D}}_4}(M_t,\,M_s)=0$ if $t\neq s$.

It is easy to see that we have the following isomorphism for $t\in k_d\setminus\{0,\,1\}$:
 \begin{center}
\begin{picture}(60,100)
  \put(73,-4){\makebox(0,0){$ k_d$.}}
\put(73,30){\makebox(0,0){$ k_d$}}
\put(73,63){\makebox(0,0){$ k_d$}}
\put(73,97){\makebox(0,0){$ k_d$}}
\put(-23,49){\makebox(0,0){$M_t\cong$}}
  \put(0,49){\makebox(0,0){$ k_d^2$}}
  \put(20,84){\makebox(0,0){$\scriptstyle(0\,1)$}}
  \put(33,62){\makebox(0,0){$\scriptstyle(1\,1)$}}
  \put(43,45){\makebox(0,0){$\scriptstyle(1\,1- t)$}}
  \put(51,20){\makebox(0,0){$\scriptstyle(1\,0)$}}
  \put(10,58){\vector(4,3){53}}
  \put(10,49){\vector(4,1){53}}
  \put(10,44){\vector(4,-1){53}}
  \put(10,38){\vector(4,-3){53}}
\end{picture}
\end{center}
Moreover, by definition,
\begin{eqnarray}\label{1.6.1}
 %\begin{center}
\begin{picture}(60,100)
  \put(73,-4){\makebox(0,0){$ k_d$,}}
\put(73,30){\makebox(0,0){$ k_d$}}
\put(73,63){\makebox(0,0){$ k_d$}}
\put(73,97){\makebox(0,0){$ k_d$}}
\put(-25,49){\makebox(0,0){$\,^{\G_d} M_t=$}}
  \put(0,49){\makebox(0,0){$ k_d^2$}}
  \put(20,84){\makebox(0,0){$\scriptstyle(0\,1)$}}
  \put(33,62){\makebox(0,0){$\scriptstyle(1\,1)$}}
  \put(43,45){\makebox(0,0){$\scriptstyle(1\,\Si_d(t))$}}
  \put(51,20){\makebox(0,0){$\scriptstyle(1\,0)$}}
  \put(10,58){\vector(4,3){53}}
  \put(10,49){\vector(4,1){53}}
  \put(10,44){\vector(4,-1){53}}
  \put(10,38){\vector(4,-3){53}}
\end{picture}
%\end{center}
\end{eqnarray}
 where  $\Si_d\in\Gal(k_d,\, k)$ is  defined by $\E^\frac{1}{d}\mapsto\zeta_d\E^\frac{1}{d}\mbox{, with }\zeta_d=e^\frac{2\pi i}{d}$  a  primitive $d$-th root of unity. 
Therefore, for $t\in k_d\setminus\{0,\,1\}$, we have the isomorphism  $\,^{\G_d} M_t\cong M_{1-\Si_d(t)}$. 
Moreover, let $\tilde{\Si}\colon k_d\lra k_d$  be the function defined by $t\mapsto 1-t$. We can prove by induction that  for $m\in\N$, we have the following isomorphism: 
 $$  \,^{\G_d^m}M_t\cong M_{\Si_d^m\circ\tilde{\Si}^m(t)}.$$
 Hence, we are interested in $t\in k_d\setminus\{0,\,1\}$ that satisfy the following conditions:
  \begin{eqnarray}\label{1.2}
t=\Si_d^n\circ\tilde{\Si}^n(t),
\end{eqnarray}
and for $1\leq m\leq n-1$,
\begin{eqnarray}\label{1.3}
t\neq\Si_d^m\circ\tilde{\Si}^m(t).
\end{eqnarray}

For $n\in\N$ and $d=\mcm(4,n)$, every  $t\in k_d$ can be written as:
 $$t=\left\{\begin{array}{ll}t_0+t_1\E^\frac{1}{d}+t_2\E^\frac{2}{d}+t_3\E^\frac{3}{d}&\mbox{ if }n\mbox{ is odd, with }t_j\in k_n\\t_0+t_1\E^\frac{1}{d}&\mbox{ if }n\equiv2(\modu4)\mbox{, with }t_j\in k_n\\t&\mbox{ if }n\equiv0(\modu4).\end{array}\right.$$
We have $\Si_d|_{k_n}=\Si_n$ and thus, $\Si_d^n(\E^\frac{1}{n})=\E^\frac{1}{n}$. Let $t\in k_d$ have the following form:
 $$t=\left\{\begin{array}{ll}\frac{1}{2}+s'\E^\frac{2}{d}&\mbox{ if }n\mbox{ is odd,}\\s'&\mbox{ if }n\mbox{ is even,}\end{array}\right.$$
for some $s'\in k_n$. Then, $t$ satisfies (\ref{1.2}).  Thus, we can actually write $t$ as 
$$t=\frac{1}{2}+s\E^\frac{1}{2},\mbox{where }s=\left\{\begin{array}{ll}\frac{s'}{\E}\E^\frac{n+1}{2n}&\mbox{ if }n\mbox{ is odd,}\\\frac{2s'-1}{2\E}\E^\frac{1}{2}&\mbox{ if }n\mbox{ is even,}\end{array}\right.$$
and in both  cases $s\in k_n$. For $1\leq m\leq n$,  this follows:
  $$\Si_d^m\circ\tilde{\Si}^m(t)=\frac{1}{2}+\Si_d^m(s)\E^\frac{1}{2}=\frac{1}{2}+\Si_n^m(s)\E^\frac{1}{2}.$$

Then, $t=\frac{1}{2}+s\E^\frac{1}{2}$ satisfies (\ref{1.3}) if $s\in k_n$ is generic. Thus, $M=\bigoplus_{j=0}^{n-1}\,^{\Si_d^j}M_{\frac{1}{2}+s\E^\frac{1}{2}}$ is ii-indecomposable for any generic $s\in k_n$.
\begin{prop} Let $M=\bigoplus_{j=0}^{n-1}\,^{\Si_d^j}M_{\frac{1}{2}+s\E^\frac{1}{2}}$ be an ii-indecomposable. Then, there exists an isomorphism $f\colon M\lra\,^{\G_d}M$ such that the automorphism $^{\G_d^{d-1}}f\circ\cdots\circ^{\G_d}f\circ f$ of M is the identity.
\end{prop}
\begin{proof}
For $x_1,\,x_2\,\cdots,\,x_n\in k_d$ and $1\leq l\leq n$, let $N_l\in\M_{n\times n}(k_d)$ be a matrix such that $(N_l)_{ij}=\left\{\begin{array}{ll}\D_{i+1,\,j}&\mbox{if }i<n\\\D_{1,\,j}x_l&\mbox{if }i=n\end{array}\right.$, 
$$N_l=\left[\begin{array}{ccccc}0&1&0&\cdots&0\\ 0&0&1&\cdots&0\\\vdots&\vdots&\vdots&\ddots&\vdots\\0&0&0&\cdots&1\\x_l&0&0&\cdots&0\end{array}\right].$$

Then, it is easy to see that
$$N_n\cdots N_2N_1=\diag(x_1,\,x_2,\cdots,\,x_n).$$
Let $X,\,Y,\,Z,\,W\in\M_{n\times n}(k_d)$ be matrices such that
$$\begin{array}{ll}(X)_{ij}\coloneqq \left\{\begin{array}{ll}\D_{i+1,\,j}&\mbox{if }i<n\\\D_{1,\,j}x&\mbox{if }i=n\end{array}\right.&(Y)_{ij}\coloneqq \left\{\begin{array}{ll}y&\mbox{if }i=n,\,j=1\\0&\mbox{otherwise}\end{array}\right.\\
(Z)_{ij}\coloneqq \left\{\begin{array}{ll}z&\mbox{if }i=n,\,j=1\\0&\mbox{otherwise}\end{array}\right.&(W)_{ij}\coloneqq \left\{\begin{array}{ll}\D_{i+1,\,j}&\mbox{if }i<n\\\D_{1,\,j}w&\mbox{if }i=n\end{array}\right.,
\end{array}$$
with $x,\,y,\,z,\,w\in k_d$. Take $\left[\begin{array}{cc}X&Y\\Z&W\end{array}\right]\in\M_{2n}(k_d)$. Then, it is easy to see that
\[
\begin{split}
\Si_d^{n-1}\left(\left[\begin{array}{cc}X&Y\\Z&W\end{array}\right]\right)\cdots\Si_d\left(\left[\begin{array}{cc}X&Y\\Z&W\end{array}\right]\right)\left[\begin{array}{cc}X&Y\\Z&W\end{array}\right]=\\\left[\begin{array}{cc}\diag(x,\,\Si_d(x),\cdots,\,\Si_d^{n-1}(x))&\diag(y,\,\Si_d(y),\cdots,\,\Si_d^{n-1}(y))\\\diag(z,\,\Si_d(z),\cdots,\,\Si_d^{n-1}(z))&\diag(w,\,\Si_d(w),\cdots,\,\Si_d^{n-1}(w))\end{array}\right]
\end{split}
\]
Now, we have to work with separate cases according to the value of $n$ ($\modu 4$):
\begin{enumerate}
\item Let $n\equiv1(\modu2)$. Then, $d=4n$. Since $s\in k_n$ thus, $t^{-1}\in k_{2n}$ with $t=\frac{1}{2}+s\E^\frac{1}{2}$. We have from Lemma \ref{cuadrado} that $t^{-1}=r^2$ or $t^{-1}=r^2\E^\frac{2}{d}$ for some $r\in k_{2n}$. Consider $a=\sqrt{i}r$ or $a=\sqrt{i}r\E^\frac{1}{d}$. Then, $a\Si_d^{2n}(a)=it^{-1}$ or $a\Si_d^{2n}(a)=-it^{-1}$. Note that in both cases, $-a\Si_d^n(a)\Si_d^{2n}(a)\Si_d^{3n}(a)t\Si_d^n(t)=1$.
\begin{enumerate}
\item Let $n\equiv 1(\modu4)$ be satisfied and $f\colon M\lra{}^{\G_d}M$ be the isomorphism defined by
$$f=\left(\begin{array}{c}A_1\\A_2\\A_3\\A_4\\\left[\begin{array}{cc}X&Y\\Z&W\end{array}\right]\end{array}\right),$$
with the matrices $A_1,\,A_2,\,A_3,\,A_4,\,X,\,Y,\,Z,\,W\in\M_{n\times n}(k_d)$ given by:
$$\begin{array}{ll}(A_1)_{ij}=\left\{\begin{array}{ll}\D_{i+1,\,j}&\mbox{if }i<n\\\D_{1,\,j}(-a)&\mbox{if }i=n\end{array}\right.,&(A_2)_{ij}=\left\{\begin{array}{ll}\D_{i+1,\,j}&\mbox{if }i<n\\\D_{1,\,j}at&\mbox{if }i=n\end{array}\right.,\\
&\\
(A_3)_{ij}=\left\{\begin{array}{ll}\D_{i+1,\,j}&\mbox{if }i<n\\\D_{1,\,j}at&\mbox{if }i=n\end{array}\right.,& (A_4)_{ij}=\left\{\begin{array}{ll}\D_{i+1,\,j}&\mbox{if }i<n\\\D_{1,\,j}a&\mbox{if }i=n\end{array}\right.,
\\
&\\
(X)_{ij}=\left\{\begin{array}{ll}\D_{i+1,\,j}&\mbox{if }i<n\\\D_{1,\,j}a&\mbox{if }i=n\end{array}\right.,&(Y)_{ij}=\left\{\begin{array}{ll}at&\mbox{if }i=n,\,j=1\\0&\mbox{otherwise}\end{array}\right.,\\
&\\
(Z)_{ij}=\left\{\begin{array}{ll}-a&\mbox{if }i=n,\,j=1\\0&\mbox{otherwise}\end{array}\right.,&(W)_{ij}=\left\{\begin{array}{ll}\D_{i+1,\,j}&\mbox{if }i<n\\\D_{1,\,j}0&\mbox{if }i=n\end{array}\right.,\end{array}$$
then $^{\G_d^{n-1}}f\circ\cdots\circ^{\G_d}f\circ f$ is equal to
\begin{eqnarray*}
&&{\scriptstyle\left(\begin{array}{c}\Si_d^{n-1}(A_1)\Si^{n-2}_d(A_4)\cdots\Si^4_d(A_1)\Si^3_d(A_4)\Si^2_d(A_3)\Si_d(A_2)A_1\\\Si_d^{n-1}(A_2)\Si^{n-2}_d(A_1)\cdots\Si^4_d(A_2)\Si^3_d(A_1)\Si^2_d(A_4)\Si_d(A_3)A_2\\\Si_d^{n-1}(A_3)\Si^{n-2}_d(A_2)\cdots\Si^4_d(A_3)\Si^3_d(A_2)\Si^2_d(A_1)\Si_d(A_4)A_3\\\Si_d^{n-1}(A_4)\Si^{n-2}_d(A_3)\cdots\Si^4_d(A_4)\Si^3_d(A_3)\Si^2_d(A_2)\Si_d(A_1)A_4\\\Si_d^{n-1}\left(\left[\begin{array}{cc}X&Y\\Z&W\end{array}\right]\right)\cdots\Si_d\left(\left[\begin{array}{cc}X&Y\\Z&W\end{array}\right]\right)\left[\begin{array}{cc}X&Y\\Z&W\end{array}\right]\end{array}\right)}\\
&=&{\scriptstyle\left(\begin{array}{c}\diag(-a,\,\Si_d(at),\,\Si^2_d(at),\,\cdots,\,\Si^{n-2}_d(a),\,\Si^{n-1}_d(-a))\\\diag(at,\,\Si_d(at),\,\Si^2_d(a),\,\cdots,\,\Si^{n-2}_d(-a),\,\Si^{n-1}_d(at))\\\diag(at,\,\Si_d(a),\,\Si^2_d(-a),\,\cdots,\,\Si^{n-2}_d(at),\,\Si^{n-1}_d(at))\\\diag(a,\,\Si_d(-a),\,\Si^2_d(at),\,\cdots,\,\Si^{n-2}_d(at),\,\Si^{n-1}_d(a))\\\left[\begin{array}{cc}\diag(a,\cdots,\,\Si_d^{n-1}(a))&\diag(at,\cdots,\,\Si_d^{n-1}(at))\\\diag(-a,\,\cdots,\,\Si_d^{n-1}(-a))&0_n\end{array}\right]\end{array}\right)}
\end{eqnarray*}
$^{\G_d^{2n-1}}f\circ\cdots\circ\,^{\G_d^{n+1}}f\circ \,^{\G_d^n}f=$
\begin{eqnarray*}
\left(\begin{array}{c}\diag(\Si_d^n(at),\,\Si_d^{n+1}(at),\,\cdots,\,\Si^{2n-2}_d(-a),\,\Si^{2n-1}_d(at))\\\diag(\Si_d^n(at),\,\Si_d^{n+1}(a),\,\cdots,\,\Si^{2n-2}_d(at),\,\Si^{2n-1}_d(at))\\\diag(\Si_d^n(a),\,\Si_d^{n+1}(-a),\,\cdots,\,\Si^{2n-2}_d(at),\,\Si^{2n-1}_d(a))\\\diag(\Si_d^n(-a),\,\Si_d^{n+1}(at),\,\cdots,\,\Si^{2n-2}_d(a),\,\Si^{2n-1}_d(-a))\\\left[\begin{array}{cc} \diag(\Si_d^n(a),\cdots,\,\Si_d^{2n-1}(a))&\diag(\Si_d^n(at),\cdots,\,\Si_d^{2n-1}(at))\\\diag(\Si_d^n(-a),\,\cdots,\,\Si_d^{2n-1}(-a))&0_n\end{array}\right]\end{array}\right),
\end{eqnarray*}
$^{\G_d^{3n-1}}f\circ\cdots\circ^{\G_d^{2n+1}}f\circ \,^{\G_d^{2n}}f=$
\begin{eqnarray*}
\left(\begin{array}{c}\diag(\Si_d^{2n}(at),\,\Si_d^{2n+1}(a),\,\cdots,\,\Si^{3n-2}_d(at),\,\Si^{3n-1}_d(at))\\\diag(\Si_d^{2n}(a),\,\Si_d^{2n+1}(-a),\,\cdots,\,\Si^{3n-2}_d(at),\,\Si^{3n-1}_d(a))\\\diag(\Si_d^{2n}(-a),\,\Si_d^{2n+1}(at),\,\cdots,\,\Si^{3n-2}_d(a),\,\Si^{3n-1}_d(-a))\\\diag(\Si_d^{2n}(at),\,\Si_d^{2n+1}(at),\,\cdots,\,\Si^{3n-2}_d(-a),\,\Si^{3n-1}_d(at))\\\left[\begin{array}{cc} \diag(\Si_d^{2n}(a),\cdots,\,\Si_d^{3n-1}(a))&\diag(\Si_d^{2n}(at),\cdots,\,\Si_d^{3n-1}(at))\\\diag(\Si_d^{2n}(-a),\,\cdots,\,\Si_d^{3n-1}(-a))&0_n\end{array}\right]\end{array}\right),
\end{eqnarray*}
$^{\G_d^{4n-1}}f\circ\cdots\circ^{\G_d^{3n+1}}f\circ \,^{\G_d^{3n}}f=$
\begin{eqnarray*}
\left(\begin{array}{c}\diag(\Si_d^{3n}(a),\,\Si_d^{3n+1}(-a),\,\cdots,\,\Si^{4n-2}_d(at),\,\Si^{4n-1}_d(a))\\\diag(\Si_d^{3n}(-a),\,\Si_d^{3n+1}(at),\,\cdots,\,\Si^{4n-2}_d(a),\,\Si^{4n-1}_d(-a))\\\diag(\Si_d^{3n}(at),\,\Si_d^{3n+1}(at),\,\cdots,\,\Si^{4n-2}_d(-a),\,\Si^{4n-1}_d(at))\\\diag(\Si_d^{3n}(at),\,\Si_d^{3n+1}(a),\,\cdots,\,\Si^{4n-2}_d(at),\,\Si^{4n-1}_d(at))\\\left[\begin{array}{cc} \diag(\Si_d^{3n}(a),\cdots,\,\Si_d^{4n-1}(a))&\diag(\Si_d^{3n}(at),\cdots,\,\Si_d^{4n-1}(at))\\\diag(\Si_d^{3n}(-a),\,\cdots,\,\Si_d^{4n-1}(-a))&0_n\end{array}\right]\end{array}\right).
\end{eqnarray*}
After some computations, one can see that $^{\G_d^{4n-1}}f\circ\cdots\circ\,^{\G_d}f\circ f$ is equal to the diagonal matrix $\diag(d_1,\cdots,\,d_{4n-1})$ with 
$$d_j\Si_d^j(-a\Si_d^n(at)\Si_d^{2n}(at)\Si_d^{3n}(a))=\Si_d^j(-a\Si_d^n(a)\Si_d^{2n}(a)\Si_d^{3n}(a)t\Si_d^n(t)).$$ 
 But we saw at the beginning  of case a) that $-a\Si_d^n(a)\Si_d^{2n}(a)\Si_d^{3n}(a)t\Si_d^n(t)=1$. Therefore $^{\G_d^{4n-1}}f\circ\cdots\circ\,^{\G_d}f\circ f$ is the identity. 
\item Let $n\equiv 3(\modu4)$. Let $f\colon M\lra^{\G_d}M$ be the isomorphism defined by
$$f=\left(\begin{array}{c}A_1\\A_2\\A_3\\A_4\\\left[\begin{array}{cc}X&Y\\Z&W\end{array}\right]\end{array}\right),$$
with $A_1,\,A_2,\,A_3,\,A_4,\,X,\,Y,\,Z,\,W\in\M_{n\times n}(k_d)$ described by:
$$\begin{array}{ll}(A_1)_{ij}=\left\{\begin{array}{ll}\D_{i+1,\,j}&\mbox{if }i<n\\\D_{1,\,j}a\Si_d^n(t)&\mbox{if }i=n\end{array}\right.,&(A_2)_{ij}=\left\{\begin{array}{ll}\D_{i+1,\,j}&\mbox{if }i<n\\\D_{1,\,j}(-a)\Si_d^n(t)&\mbox{if }i=n\end{array}\right.,\\
&\\
(A_3)_{ij}=\left\{\begin{array}{ll}\D_{i+1,\,j}&\mbox{if }i<n\\\D_{1,\,j}a&\mbox{if }i=n\end{array}\right.,& (A_4)_{ij}=\left\{\begin{array}{ll}\D_{i+1,\,j}&\mbox{if }i<n\\\D_{1,\,j}a&\mbox{if }i=n\end{array}\right.,\\
&\\
(X)_{ij}=\left\{\begin{array}{ll}\D_{i+1,\,j}&\mbox{if }i<n\\\D_{1,\,j}0&\mbox{if }i=n\end{array}\right.,&(Y)_{ij}=\left\{\begin{array}{ll}a\Si_d^n(t)&\mbox{if }i=n,\,j=1\\0&\mbox{otherwise}\end{array}\right.,\\
&\\
(Z)_{ij}=\left\{\begin{array}{ll}a&\mbox{if }i=n,\,j=1\\0&\mbox{otherwise}\end{array}\right.,&(W)_{ij}=\left\{\begin{array}{ll}\D_{i+1,\,j}&\mbox{if }i<n\\\D_{1,\,j}a&\mbox{if }i=n.\end{array}\right.\end{array}$$
The process to show $^{\G_d^{4n-1}}f\circ\cdots\circ\,^{\G_d}f\circ f$ is the identity, is analogous to that (i) and will be omitted here.
\end{enumerate}
\item Let $n\equiv2(\modu4)$. Then, $d=\mcm(4,n)=2n$. Let $t=\frac{1}{2}+s\E^\frac{1}{2}$ with $s\in k_n$.  Thus, $(t(1-t))^{-1}=(\frac{1}{4}-s^2\E)^{-1}\in k_{n}$. We have from Lemma \ref{cuadrado} that $(t-t^2)^{-1}=r^2$ or $(t-t^2)^{-1}=r^2\E^\frac{1}{n}$ for some $r\in k_{n}$. 

If we take $a=ir$ or $a=r\E^\frac{1}{d}$, then $-a\Si_d^{n}(a)=(t-t^2)^{-1}$. Thus, $-a\Si_d^n(a)t(1-t)$ is equal to $1$.
Let $f\colon M\lra^{\G_d}M$ be the isomorphism defined by
$$f=\left(\begin{array}{c}A_1\\A_2\\A_3\\A_4\\\left[\begin{array}{cc}X&Y\\Z&W\end{array}\right]\end{array}\right),$$
with $A_1,\,A_2,\,A_3,\,A_4,\,X,\,Y,\,Z,\,W\in\M_{n\times n}(k_d)$ described by:
$$\begin{array}{ll}(A_1)_{ij}=\left\{\begin{array}{ll}\D_{i+1,\,j}&\mbox{if }i<n\\\D_{1,\,j}a(1-t)&\mbox{if }i=n\end{array}\right.,& (A_2)_{ij}=\left\{\begin{array}{ll}\D_{i+1,\,j}&\mbox{if }i<n\\\D_{1,\,j}(-a)(t-t^2)&\mbox{if }i=n\end{array}\right.,\\
&\\
(A_3)_{ij}=\left\{\begin{array}{ll}\D_{i+1,\,j}&\mbox{if }i<n\\\D_{1,\,j}(-at)&\mbox{if }i=n\end{array}\right.,& (A_4)_{ij}=\left\{\begin{array}{ll}\D_{i+1,\,j}&\mbox{if }i<n\\\D_{1,\,j}a&\mbox{if }i=n\end{array}\right.,\\
&\\
(X)_{ij}=\left\{\begin{array}{ll}\D_{i+1,\,j}&\mbox{if }i<n\\\D_{1,\,j}(-at)&\mbox{if }i=n\end{array}\right.,&(Y)_{ij}=\left\{\begin{array}{ll}-at&\mbox{if }i=n,\,j=1\\0&\mbox{otherwise}\end{array}\right.,\\
&\\
(Z)_{ij}=\left\{\begin{array}{ll}a&\mbox{if }i=n,\,j=1\\0&\mbox{otherwise}\end{array}\right.,&(W)_{ij}=\left\{\begin{array}{ll}\D_{i+1,\,j}&\mbox{if }i<n\\\D_{1,\,j}at&\mbox{if }i=n.\end{array}\right.\end{array}$$
As above, after some computation one can see that $^{\G_d^{2n-1}}f\circ\cdots\circ\,^{\G_d}f\circ f$ is equal to the diagonal matrix with entries $\Si_d^j(-a\Si_d^n(a)t(1-t))$  for $1\leq j\leq 2n-1$. Therefore $^{\G_d^{2n-1}}f\circ\cdots\circ\,^{\G_d}f\circ f$ is the identity.
\item Let $n\equiv0(\modu4)$ be satisfied and $f\colon M\lra^{\G_d}M$ be the isomorphism defined by
$$f=\left(\begin{array}{c}A_1\\A_1\\A_1\\A_1\\\left[\begin{array}{cc}A_3&A_2\\A_2&A_3\end{array}\right]\end{array}\right),$$
with $A_1,\,A_2,\,A_3\in\M_{n\times n}(k_d)$ described by:
$$\begin{array}{ll}(A_1)_{ij}=\left\{\begin{array}{ll}\D_{i+1,\,j}&\mbox{if }i<n\\\D_{1,\,j}&\mbox{if }i=n\end{array}\right., &(A_2)_{ij}=\left\{\begin{array}{ll}1&\mbox{if }i=n,\,j=1\\0&\mbox{otherwise}\end{array}\right.,\\
(A_3)_{ij}=\left\{\begin{array}{ll}\D_{i+1,\,j}&\mbox{if }i<n\\\D_{1,\,j}0&\mbox{if }i=n.\end{array}\right.&\end{array}$$
It is easy to see that $^{\G_d^{d-1}}f\circ\cdots\circ\,^{\G_d}f\circ f$ is the identity.
\end{enumerate}
\end{proof}
Finally, we have two ii-simple regular representations that we have not yet mentioned. They are the following:
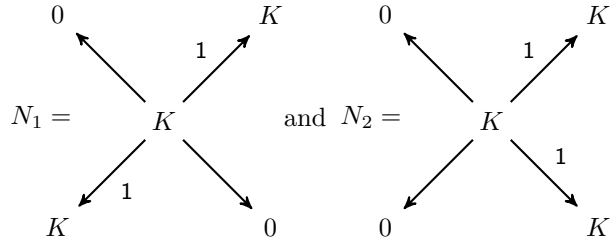
\begin{center}
 \begin{picture}(170,105)
 \put(0,49){\makebox(0,0){$N_1=$}}
 \put(10,0){\makebox(70,95){
\begin{tikzpicture}
[->,>=stealth',shorten >=1pt,auto,node distance=2cm,thick,main node/.style=]
  \node[main node] (1) {$0$};   
  \node[main node] (2) [below right of=1] {$K$};
  \node[main node] (3) [below left of=2] {$K$};
  \node[main node] (4) [above right of=2] {$K$};
  \node[main node] (5) [below right of=2] {$0$};
\path[every node/.style={font=\sffamily\small}]     
(2) edge node   {} (1)         
    edge node         {1} (4)             
  edge node  {1} (3)              
   edge node         {} (5)
; 
\end{tikzpicture}}} 
 \put(100,49){\makebox(0,0){and}}
 \put(125,49){\makebox(0,0){$N_2=$}}
 \put(135,0){\makebox(70,95){
\begin{tikzpicture}
[->,>=stealth',shorten >=1pt,auto,node distance=2cm,thick,main node/.style=]
  \node[main node] (1) {$0$};   
  \node[main node] (2) [below right of=1] {$K$};
  \node[main node] (3) [below left of=2] {$0$};
  \node[main node] (4) [above right of=2] {$K$};
  \node[main node] (5) [below right of=2] {$K$,};
\path[every node/.style={font=\sffamily\small}]     
(2) edge node   {} (1)         
    edge node         {1} (4)             
  edge node  {} (3)              
   edge node         {1} (5)
; 
\end{tikzpicture}}} 
\end{picture}
 \end{center}
with $K=k_4$. Because the morphisms are the identity or the zero map, we have that the  corresponding $\G\coloneqq  \G_4$ only swaps the  ``arms'' in each representation. Then,
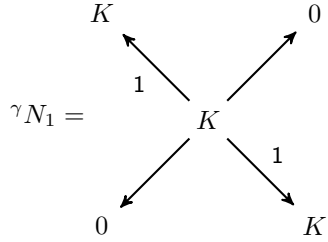
\begin{center}
 \begin{picture}(70,95)
  \put(-15,49){\makebox(0,0){$\,^\G N_1=$}}
 \put(10,0){\makebox(70,95){
\begin{tikzpicture}
[->,>=stealth',shorten >=1pt,auto,node distance=2cm,thick,main node/.style=]
  \node[main node] (1) {$K$};   
  \node[main node] (2) [below right of=1] {$K$};
  \node[main node] (3) [below left of=2] {$0$};
  \node[main node] (4) [above right of=2] {$0$};
  \node[main node] (5) [below right of=2] {$K$};
\path[every node/.style={font=\sffamily\small}]     
(2) edge node   {1} (1)         
    edge node         {} (4)             
  edge node  {} (3)              
   edge node         {1} (5)
; 
\end{tikzpicture}}} \end{picture}
 \end{center}
 \vspace{0.5cm}
 and
 
 
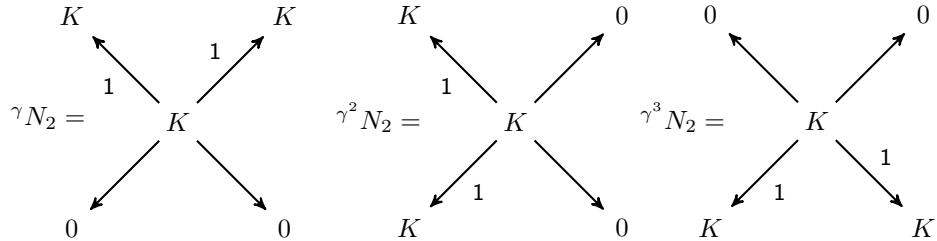
\begin{center}
  \begin{picture}(300,105)
  \put(0,49){\makebox(0,0){$\,^\G N_2=$}}
 \put(15,0){\makebox(70,95){\begin{tikzpicture}
[->,>=stealth',shorten >=1pt,auto,node distance=2cm,thick,main node/.style=]
  \node[main node] (1) {$K$};   
  \node[main node] (2) [below right of=1] {$K$};
  \node[main node] (3) [below left of=2] {$0$};
  \node[main node] (4) [above right of=2] {$K$};
  \node[main node] (5) [below right of=2] {$0$};
\path[every node/.style={font=\sffamily\small}]     
(2) edge node   {1} (1)         
    edge node         {1} (4)             
  edge node  {} (3)              
   edge node         {} (5)
; 
\end{tikzpicture}}}
 \put(125,49){\makebox(0,0){$\,^{\G^2} N_2=$}}
 \put(140,0){\makebox(70,95){ 
 \begin{tikzpicture}
[->,>=stealth',shorten >=1pt,auto,node distance=2cm,thick,main node/.style=]
  \node[main node] (1) {$K$};   
  \node[main node] (2) [below right of=1] {$K$};
  \node[main node] (3) [below left of=2] {$K$};
  \node[main node] (4) [above right of=2] {$0$};
  \node[main node] (5) [below right of=2] {$0$};
\path[every node/.style={font=\sffamily\small}]     
(2) edge node   {1} (1)         
    edge node         {} (4)             
  edge node  {1} (3)              
   edge node         {} (5)
; 
\end{tikzpicture}}}
\put(240,49){\makebox(0,0){$\,^{\G^3} N_2=$}}
 \put(255,0){\makebox(70,95){
 \begin{tikzpicture}
[->,>=stealth',shorten >=1pt,auto,node distance=2cm,thick,main node/.style=]
  \node[main node] (1) {$0$};   
  \node[main node] (2) [below right of=1] {$K$};
  \node[main node] (3) [below left of=2] {$K$};
  \node[main node] (4) [above right of=2] {$0$};
  \node[main node] (5) [below right of=2] {$K$.};
\path[every node/.style={font=\sffamily\small}]     
(2) edge node   {} (1)         
    edge node         {} (4)             
  edge node  {1} (3)              
   edge node         {1} (5)
; 
\end{tikzpicture}}}
 \end{picture}
\end{center}
Clearly, $^{\G^2}N_1=N_1$ and $^{\G^4}N_2=N_2$. Therefore, $N_1\oplus\,^{\G}N_1\mbox{ and }N_2\oplus\,^{\G}N_2\oplus\,^{\G^2}N_2\oplus\,^{\G^3}N_2$ are ii-indecomposable.
\begin{prop} \label{prop2.2}Let $M$ be equal to $N_1\oplus\,^{\G}N_1\mbox{ or }M=N_2\oplus\,^{\G}N_2\oplus\,^{\G^2}N_2\oplus\,^{\G^3}N_2$. Then, there exists an isomorphism $f\colon M\lra\,^{\G}M$ such that the automorphism $\,^{\G}f\circ f$ or $\,^{\G^{2}}f\circ\,^{\G^{2}}f\circ\,^{\G}f\circ f$ of M is the identity, respectively.
\end{prop}
\begin{proof}
It is easy to verify that
$$f=\left(\begin{array}{c}\id_K\\\id_K\\\id_K\\\id_K\\\left[\begin{array}{cc}0&1\\1&0\end{array}\right]\end{array}\right)\mbox{, }\,f=\left(\begin{array}{c}\left[\begin{array}{cc}1&0\\0&1\end{array}\right]\\\left[\begin{array}{cc}1&0\\0&1\end{array}\right]\\\left[\begin{array}{cc}1&0\\0&1\end{array}\right]\\\left[\begin{array}{cc}1&0\\0&1\end{array}\right]\\\left[\begin{array}{cccc}0&1&0&0\\0&0&1&0\\0&0&0&1\\1&0&0&0\end{array}\right]\end{array}\right),$$
has the desired property, for $M=N_1\oplus\,^{\G}N_1\mbox{ and }M=N_2\oplus\,^{\G}N_2\oplus\,^{\G^2}N_2\oplus\,^{\G^3}N_2$, respectively.
\end{proof}
\begin{prop}\label{prop2.3}
Let $\varphi_M\colon K\otimes_k k^{2n}\lra K^n$ be a simple regular representation. If $X$ is a  representation of  $\widetilde{\mathsf{D}}_4$ such that $k_d\otimes_k M\cong X$, then there  exists an isomorphism $f\colon X\lra\,^{\G_d}X$ such that the automorphism $^{\G_d^{m-1}}f\circ\cdots\circ^{\G_d}f\circ f$ of $X$ is the identity, with $m=\left\{\begin{array}{ll}d&\mbox{if }\End(M)\cong k_n\\2n&\mbox{if }\End(M)\cong k_{2n}\end{array}\right.$
and $d=\mcm(4,n)$.
\end{prop}
\begin{proof}Let $M$ be a simple regular representation of $\Lambda$ with dimension vector $(n,\,2n)$. As we saw at the beginning of this section, we can think of $k_d\otimes_k M$ as the following representation of the quiver $\widetilde{\mathsf{D}}_4$:
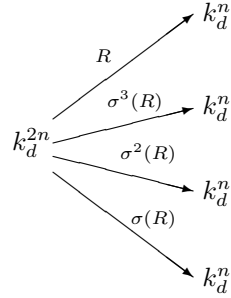
\begin{center}
 \begin{picture}(100,100)
  \put(85,-3){\makebox(0,0){$  k_d^n$}}
\put(85,30){\makebox(0,0){$  k_d^n$}}
\put(85,62){\makebox(0,0){$  k_d^n$}}
\put(85,97){\makebox(0,0){$  k_d^n$}}
  \put(15,49){\makebox(0,0){$  k_d^{2n}$}}
  \put(42,82){\makebox(0,0){$\scriptstyle R$}}
  \put(54,65){\makebox(0,0){$\scriptstyle\Si^3(R)$}}
  \put(59,45){\makebox(0,0){$\scriptstyle\Si^2(R)$}}
  \put(61,20){\makebox(0,0){$\scriptstyle\Si(R)$}}
  \put(23,58){\vector(4,3){53}}
  \put(23,49){\vector(4,1){53}}
  \put(23,44){\vector(4,-1){53}}
  \put(23,38){\vector(4,-3){53}}
\end{picture}
\end{center}
where $R$ is the matrix of $\varphi_M$ and $\Si=\Si_4$.

Since $\Si_d|_K=\Si$  and $R\in\M_{n\times 2n}(K)$, we have $\Si_d(\Si^j(R))=\Si(\Si^j(R))=\Si^{j+1}(R)$, for $0\leq j\leq 3$. Then,  $k_d\otimes_k M=\,^{\G_d}(k_d\otimes_k M)$.

Now, let $g\colon k_d\otimes_k M\lra X$ be a $k_d$-linear isomorphism of representations of $\widetilde{\mathsf{D}}_4$. Then, $\,^{\G_d}g\colon\,^{\G_d}(k_d\otimes_k M)\lra \,^{\G_d}X$ is also an isomorphism. Thus, 
$$f\coloneqq \,^{\G_d}g\circ g^{-1}\colon X\lra\,^{\G_d}X$$
 is an isomorphism. Therefore, $^{\G_d^{m-1}}f\circ\cdots\circ^{\G_d}f\circ f=\,^{\G_d^{m}}g\circ g^{-1}=\id$ since $^{\G_d^{m}}g=g$.
\end{proof}
%%%%%%%%%%%%%%%%%%%%%%%%%%%%%%%%%%%%%%%%%%%%%%%%%%%%%%%%%%%%%%%%%%%%%%%%%%%%%%%%%%%%%%%%%%%%%%%%%%%%%%%%%%%%%%%%%%%%%%%%%%%%%%%%%%%%%%%%%%%%%%%%%%%%%%%%%%%%%%%%%%%%%%%%%%%%%%%%%%%%%%%%%%%%%%%%%%%%%%%%%%%%%%%%%%%%%%%%%%%
\subsection{A classification of simple regular representations}\label{2.3}
Recall that $K[\A_n]$ is a commutative $K$-algebra. More specifically, we have the following result.
\begin{lema}\label{2.3.1}For $n\in\N$, we have the following isomorphisms of $K$-algebras.
\begin{enumerate}
\item If $n\equiv1(\modu2)$, then $K[\A_n]\cong k_{4n}$.
\item If $n\equiv2(\modu4)$, then $K[\A_n]\cong k_{2n}\times k_{2n}$.
\item If $n\equiv0(\modu4)$, then $K[\A_n]\cong k_{n}\times k_{n}\times k_{n}\times k_{n}$.
\end{enumerate}
\end{lema}
\begin{proof}
\begin{enumerate}
\item Let $n\in\Z_{>0}$ be such that $n\equiv1(\modu2)$. It is enough to take the isomorphism $\E^\frac{1}{4n}\mapsto \frac{\E^\frac{1}{4}}{\E}\A_n^\frac{3n+1}{4}$ and $\E^\frac{1}{4n}\mapsto \frac{\E^\frac{3}{4}}{\E}\A_n^\frac{n+1}{4}$ for $n\equiv 1(\modu 4)$ and  $n\equiv 3(\modu 4)$, respectively.
\item Let $n\in\Z_{>0}$ be such that $n\equiv 2(\modu 4)$. It is enough to show that $K[\A_n]$ is isomorphic to $K[\A_\frac{n}{2}]\times K[\A_\frac{n}{2}]$ because $\frac{n}{2}$ is odd. This is easily done via the isomorphism  $\A_n\mapsto(\A_\frac{n}{2}^\frac{n+2}{4}\frac{\E^\frac{1}{2}}{\E},\,-\A_\frac{n}{2}^\frac{n+2}{4}\frac{\E^\frac{1}{2}}{\E})$. 
\item Let $n\equiv0(\modu4)$. Then, the isomorphism we are looking for  between $K[\A_n]$ and $k_{n}\times k_{n}\times k_{n}\times k_{n}$ is the one that
maps $\A_n$ to $(\zeta_n^4\E^\frac{1}{n},\,\zeta_n^3\E^\frac{1}{n},\,\zeta_n^2\E^\frac{1}{n},\,\zeta_n\E^\frac{1}{n})$.
\end{enumerate}
\end{proof}
%  On the other hand, recall that $\Si_n$ is an automorphism of fields defined by  $\E^\frac{1}{n}\mapsto \zeta_n\E^\frac{1}{n}$, as in Section \ref{quasi1.2}. We have the following  result:

Now, recall that the automorphism $\Si\colon  K\lra K$ is  defined by  $\E^\frac{1}{4}\mapsto i\E^\frac{1}{4}$. Consider $R=\displaystyle\sum_{j=0}^3\A(\nu_j)\E^\frac{j}{4}\in K[\A_n]$. Then, $\Si$ acts on $R$ by $\Si(R)=\displaystyle\sum_{j=0}^3\A(\nu_j)\Si(\E^\frac{j}{4})$.  Thus, it is easy to see that we have the following equalities 
\begin{eqnarray*}
\Si^2(R)-R&=&-2\E^\frac{1}{4}(\A(\nu_1)+\E^\frac{1}{2}\A(\nu_3)),\\
\Si^3(R)-\Si(R)&=&-2i\E^\frac{1}{4}(\A(\nu_1)-\E^\frac{1}{2}\A(\nu_3)),\\
%\end{eqnarray*}
%\begin{eqnarray*}
\Si(R)-R&=&(i-1)\A(\nu_1)\E^\frac{1}{4}-2\A(\nu_2)\E^\frac{1}{2}-(1+i)\A(\nu_3)\E^\frac{3}{4},\\
\Si^3(R)-\Si^2(R)&=&-(i-1)\A(\nu_1)\E^\frac{1}{4}-2\A(\nu_2)\E^\frac{1}{2}+(1+i)\A(\nu_3)\E^\frac{3}{4}.\\
\end{eqnarray*}
Therefore,
\begin{eqnarray}\label{uw}
(\Si^2(R)-R)(\Si^3(R)-\Si(R))&=&4i\E^\frac{1}{2}\left(\A(\nu_1^2)-\E \A(\nu_3^2)\right)
\end{eqnarray}
and
\begin{eqnarray}\label{v}
(\Si(R)-R)(\Si^3(R)-\Si^2(R))=4\E\left(\A(\nu_2^2)-\A(\nu_1)\A(\nu_3)\right)+2i\E^\frac{1}{2}\left(\A(\nu_1^2)-\E \A(\nu_3^2)\right)
\end{eqnarray}
%}
Now, for $R\in K[\A_n]_{reg}$, using equations  \eqref{uw} and \eqref{v}, we get that

 \begin{equation}\label{2.6}
 \begin{split}
 (\Si(R)-R)(\Si^3(R)-\Si^2(R))&[(\Si^2(R)-R)(\Si^3(R)-\Si(R))]^{-1}\\
 =\frac{1}{2}\id_n-i\E^\frac{1}{2}\left(\A(\nu_2^2)-\A(\nu_1)\right.&\left.\A(\nu_3)\right)\left(\A(\nu_1^2)-\E \A(\nu_3^2)\right)^{-1}\end{split}
 \end{equation}
The Equation \ref{2.6} will be used to prove Proposition \ref{prop2.8}.
 \begin{remark}\label{remark2.4.3}{$R\in K[\A_n]_{\reg}$ if and only if $\Si(R)-R$  and $\Si^2(R)-R$ are invertible matrices.}
\end{remark}
\begin{proof}Let $A$ be as in Lemma  \ref{lema2}. Then,
\begin{eqnarray*}
A^{-1}(\Si(R)-R)A&=&A^{-1} ((i-1)\A(\nu_1)\E^\frac{1}{4}-2\A(\nu_2)\E^\frac{1}{2}-(1+i)\A(\nu_3)\E^\frac{3}{4})A\\
&=&(i-1)A^{-1}\A(\nu_1)A\E^\frac{1}{4}-2A^{-1}\A(\nu_2)A\E^\frac{1}{2}-(1+i)A^{-1}\A(\nu_1)A\E^\frac{3}{4}.
\end{eqnarray*}
By Lemma \ref{lema2} we have each entry in the diagonal is as follows:
$$(i-1)\Si_n^{m}(\nu_1)\E^\frac{1}{4}-2\Si_n^{m}(\nu_2)\E^\frac{1}{2}-(1+i)\Si_n^{m}(\nu_3)\E^\frac{3}{4}$$
with $1\leq m\leq n$. 

Therefore, $\Si(R)-R$ is invertible if and only if  
$$(i-1)\Si_n^{m}(\nu_1)\E^\frac{1}{4}-2\Si_n^{m}(\nu_2)\E^\frac{1}{2}-(1+i)\Si_n^{m}(\nu_3)\E^\frac{3}{4}\neq 0$$
  for $1\leq m\leq n$, i.e. if and only if $R$ satisfies (R2), see Definition \ref{def:regular}.
  
  Similarly, $\Si^2(R)-R$ is invertible if and only if $R$ satisfies (R1).
\end{proof}
\begin{remark}\label{remark244}{Let $R$ be a matrix in $K[\A_n]$ be described as
$$R=\A(\nu_0)+\A(\nu_1)\E^\frac{1}{4}+\A(\nu_2)\E^\frac{1}{2}+\A(\nu_3)\E^\frac{3}{4}.$$
Then, $k_d\otimes_k M(R)\cong M_{R'}$, where  $M_{R'}$ is the representation:
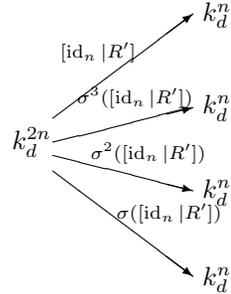
\begin{center}
 \begin{picture}(100,100)
  \put(85,-3){\makebox(0,0){$  k_d^n$}}
\put(85,30){\makebox(0,0){$  k_d^n$}}
\put(85,62){\makebox(0,0){$  k_d^n$}}
\put(85,97){\makebox(0,0){$  k_d^n$}}
  \put(15,49){\makebox(0,0){$  k_d^{2n}$}}
  \put(40,83){\makebox(0,0){$\scriptstyle [\id_n|R']$}}
  \put(54,66){\makebox(0,0){$\scriptstyle\Si^3([\id_n|R'])$}}
  \put(59,45){\makebox(0,0){$\scriptstyle\Si^2([\id_n|R'])$}}
  \put(67,22){\makebox(0,0){$\scriptstyle\Si([\id_n|R'])$}}
  \put(23,58){\vector(4,3){53}}
  \put(23,49){\vector(4,1){53}}
  \put(23,44){\vector(4,-1){53}}
  \put(23,38){\vector(4,-3){53}}
\end{picture}
\end{center}
with
$$R'\coloneqq\diag(\nu_0+\nu_1\E^\frac{1}{4}+\nu_2\E^\frac{1}{2}+\nu_3\E^\frac{3}{4},\cdots,\Si_n^{n-1}(\nu_0)+\Si_n^{n-1}(\nu_1)\E^\frac{1}{4}+\Si_n^{n-1}(\nu_2)\E^\frac{1}{2}+\Si_n^{n-1}(\nu_3)\E^\frac{3}{4}).$$}\end{remark}
\begin{proof}
This is easy to see using the $A$ matrix of Lemma \ref{lema2}.
\end{proof}
\begin{remark}\label{remark245}{ Let $n\in\Z_{\geq2}$ and  $R,\,S\in \M_{n\times n}K[\A_n]$ be described as follows
 $$R=\A(\nu_0)+\A(\nu_1)\E^\frac{1}{4}+\A(\nu_2)\E^\frac{1}{2}+\A(\nu_3)\E^\frac{3}{4},$$$$S=\A(\mu_0)+\A(\mu_1)\E^\frac{1}{4}+\A(\mu_2)\E^\frac{1}{2}+\A(\mu_3)\E^\frac{3}{4}.$$  If  $$\mu_0+\mu_1\E^\frac{1}{4}+\mu_2\E^\frac{1}{2}+\mu_3\E^\frac{3}{4}=\Si_n^m(\nu_0)+\Si_n^m(\nu_1)\E^\frac{1}{4}+\Si_n^m(\nu_2)\E^\frac{1}{2}+\Si_n^m(\nu_3)\E^\frac{3}{4}$$
 for some $1\leq m\leq n$, then $k_d\otimes_k M(R)\cong k_d\otimes_k M(S)$.}\end{remark}
\begin{proof}
It follows from the Remark \ref{remark244}.
\end{proof}
\begin{prop}\label{P310}{ Let $R$ be  in $K[\A_n]$ satisfy (R1) but not (R2). Then 
$$k_d\otimes_k M(R)\cong M(\A(\E^\frac{1}{2})\E^\frac{1}{4}+\A(-i)\E^\frac{3}{4})$$
or
$$k_d\otimes_k M(R)\cong\bigoplus_{j=0}^{\frac{n}{4}-1}\left(M_0\oplus M_1\oplus M_{(12\infty 3)}\oplus M_{(\infty 213)}\right),$$
where
 
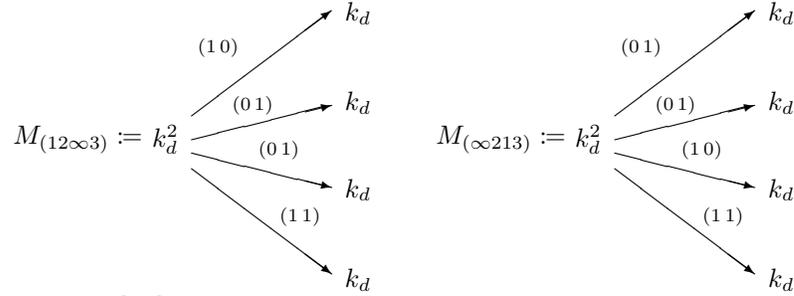
\begin{center}
\begin{picture}(250,100)
  \put(103,-4){\makebox(0,0){$ k_d$}}
\put(103,30){\makebox(0,0){$ k_d$}}
\put(103,63){\makebox(0,0){$ k_d$}}
\put(103,97){\makebox(0,0){$ k_d$}}
\put(-3,49){\makebox(0,0){$M_{(12\infty 3)}\coloneqq $}}
  \put(30,49){\makebox(0,0){$ k_d^2$}}
  \put(50,84){\makebox(0,0){$\scriptstyle(1\,0)$}}
  \put(63,62){\makebox(0,0){$\scriptstyle(0\,1)$}}
  \put(73,45){\makebox(0,0){$\scriptstyle(0\,1)$}}
  \put(81,20){\makebox(0,0){$\scriptstyle(1\,1)$}}
  \put(40,58){\vector(4,3){53}}
  \put(40,49){\vector(4,1){53}}
  \put(40,44){\vector(4,-1){53}}
  \put(40,38){\vector(4,-3){53}}
    \put(263,-4){\makebox(0,0){$ k_d$}}
\put(263,30){\makebox(0,0){$ k_d$}}
\put(263,63){\makebox(0,0){$ k_d$}}
\put(263,97){\makebox(0,0){$ k_d$}}
\put(157,49){\makebox(0,0){$M_{(\infty213)}\coloneqq $}}
  \put(190,49){\makebox(0,0){$ k_d^2$}}
  \put(210,84){\makebox(0,0){$\scriptstyle(0\,1)$}}
  \put(223,62){\makebox(0,0){$\scriptstyle(0\,1)$}}
  \put(233,45){\makebox(0,0){$\scriptstyle(1\,0)$}}
  \put(241,20){\makebox(0,0){$\scriptstyle(1\,1)$}}
  \put(200,58){\vector(4,3){53}}
  \put(200,49){\vector(4,1){53}}
  \put(200,44){\vector(4,-1){53}}
  \put(200,38){\vector(4,-3){53}}
\end{picture}
\end{center}}
\end{prop}
\begin{proof} Let $R\in K[\A_n]$ be a matrix that satisfies (R1) but not (R2), i.e. $\nu_1\neq\pm\nu_3\E^\frac{1}{2}$, and there exists $m$ in $\{0,\cdots,\,n-1\}$ such that 
 $$\Si_n^m(\nu_2)=- \frac{(1-i)\E^\frac{3}{4}}{2\E}(\Si_n^m(\nu_1)+i\Si_n^m(\nu_3)\E^\frac{1}{2}).$$
We can assume that $\nu_0=0$, because $\nu_0$ does not appear in the definition of the map $\phi_n$.
\begin{itemize}
\item Let $\nu_2=0$. Then, there exists $m$ such that $\Si_n^m(\nu_1)+i\Si_n^m(\nu_3)\E^\frac{1}{2}=0$. Using the parity of $m$, we can conclude that $\nu_1$ is equal to $i \nu_3\E^\frac{1}{2}$ or $-i \nu_3\E^\frac{1}{2}$. Let $\nu_1=i\nu_3\E^\frac{1}{2}$. Then, $R$ is equal to $\A(i\nu_3)(\A(\E^\frac{1}{2})\E^\frac{1}{4}+\A(-i)\E^\frac{3}{4})$. Therefore, $M(R)\cong M(\A(\E^\frac{1}{2})\E^\frac{1}{4}+\A(-i)\E^\frac{3}{4})$. 

Like before, we can get  the same for $\nu_1=-i\nu_3\E^\frac{1}{2}$.
\item Let $\nu_2\neq 0$. We can assume without loss of generality that \\$\nu_2= -\frac{(1-i)\E^\frac{3}{4}}{2\E}(\nu_1+i\nu_3\E^\frac{1}{2})\neq0$, and
$$R=\A(\nu_1)\E^\frac{1}{4}+\A\left(\frac{(i-1)\E^\frac{3}{4}}{2\E}(\nu_1+i\nu_3\E^\frac{1}{2})\right)\E^\frac{1}{2}+\A(\nu_3)\E^\frac{3}{4}.$$
Recall that $\A(\E^\frac{j}{4})\neq\A(1)\E^\frac{j}{4}$. For each $m\in\{0,\cdots,\,n-1\}$, we have the following:
\begin{eqnarray*}
\Si_n^m(\nu_2)&=&-\frac{(1-i)(-i)^m}{2\E}\Si_n^m(\nu_1)\E^\frac{3}{4}-\frac{(1+i)i^m}{2}\Si_n^m(\nu_3)\E^\frac{1}{4}.\\
\end{eqnarray*}
 From Remark \ref{remark244} we have $k_d\otimes_k M(R)\cong M_{R'}$. It is easy to see that each entry in the diagonal  of  $R'$ has the following form:
 $$\left(1-\frac{(1-i)(-i)^m}{2}\right)\Si_n^m(\nu_1)\E^\frac{1}{4}+\left(1-\frac{(1+i)i^m}{2}\right)\Si_n^m(\nu_3)\E^\frac{3}{4}.$$
Furthermore, each entry in the diagonal of $\Si^j(R')$ has the form
$$\left(i^j-(-1)^j\frac{(1-i)(-i)^m}{2}\right)\Si_n^m(\nu_1)\E^\frac{1}{4}+\left((-i)^j-(-1)^j\frac{(1+i)i^m}{2}\right)\Si_n^m(\nu_3)\E^\frac{3}{4},$$
where $0\leq m\leq n-1$ and $0\leq j\leq 3$. For the purpose of this proof, we will call these elements $a_{jm}$. We have 
\begin{center}
 \begin{picture}(100,110)
  \put(105,-3){\makebox(0,0){$  k_d$.}}
\put(105,30){\makebox(0,0){$  k_d$}}
\put(105,62){\makebox(0,0){$  k_d$}}
\put(105,97){\makebox(0,0){$  k_d$}}
  \put(0,49){\makebox(0,0){$ M_{R'}\cong\bigoplus_{m=0}^{n-1} k_d^{2}$}}
  \put(60,83){\makebox(0,0){$\scriptstyle [1|a_{0m}]$}}
  \put(74,66){\makebox(0,0){$\scriptstyle[1|a_{3m}]$}}
  \put(79,45){\makebox(0,0){$\scriptstyle[1|a_{2m}]$}}
  \put(87,22){\makebox(0,0){$\scriptstyle[1|a_{1m}]$}}
  \put(43,58){\vector(4,3){53}}
  \put(43,49){\vector(4,1){53}}
  \put(43,44){\vector(4,-1){53}}
  \put(43,38){\vector(4,-3){53}}
\end{picture}
\end{center}
Also, we have the following equalities
$$\begin{array}{ccc}
a_{0m}=a_{1m}&\mbox{if}&m\equiv0(\modu4),\\
a_{1m}=a_{2m}&\mbox{if}&m\equiv1(\modu4),\\
a_{2m}=a_{3m}&\mbox{if}&m\equiv2(\modu4),\\
a_{3m}=a_{0m}&\mbox{if}&m\equiv3(\modu4).
\end{array}$$
Now, it is easy to see that for $m$ congruent to $0,\,1,\,2$ and $3$ modulo $4$, each of the above summands is isomorphic to $M_0,\,M_1,\,M_{(12\infty3)}$ and $M_{(\infty213)}$, respectively. 
\end{itemize}
 \end{proof}
This Proposition implies that for $S\notin K[\A_n]_{\reg}$, then $k_d\otimes_k M(S)$ is not isomorphic to $k_d\otimes_k M(R)$ for any matrix $R$ in Table \ref{table1}, such that $\phi_n(R)$ is generic, because $M_{\frac{1}{2}-i\Si_n^j(\phi_n(R))\E^\frac{1}{2}}$ is not isomorphic to any representation in Proposition \ref{P310} for $j\in\{1,\cdots,n\}$.

Consider $R=\sum_{j=0}^4\A(\nu_j)\E^\frac{j}{4}\in K[\A_n]_{\reg}$ be such that $\phi_n(R)\in k_n$ is generic. As we saw in Section \ref{Invariantrep}, $\bigoplus_{j=0}^{n-1}\,^{\G_d^j}M_{\frac{1}{2}-i\phi_n(R)\E^\frac{1}{2}}$ is simple regular. We will denote such representations by 
$$N_{\phi_n(R)}\coloneqq \bigoplus_{j=0}^{n-1}\,^{\G_d^j}M_{\frac{1}{2}-i\phi_n(R)\E^\frac{1}{2}}\cong\bigoplus_{j=0}^{n-1}M_{\frac{1}{2}-i\Si_n^j(\phi_n(R))\E^\frac{1}{2}}.$$
 
 In the following Proposition, we will see that for $R\in K[\A_n]_{\reg}$ such that $\phi_n(R)\in k_n$ is generic, the representation $M(R)$  of $\Lambda$  whose matrix is  $(\id_n|R)$ satisfies that $k_d\otimes_k M(R)$ is isomorphic to $N_{\phi_n(R)}$ and is, therefore, simple regular.
 
\begin{prop}\label{prop2.8} Let $R\in K[\A_n]_{\reg}$ be such that $\phi_n(R)\in k_n$ is generic. Then, the representation $M(R)$ of $\Lambda$ is simple regular.
\end{prop}
\begin{proof} Let $R\in K[\A_n]_{\reg}$ be such that $\phi_n(R)\in k_n$ is generic. As in the beginning of subsection \ref{Invariantrep}, one can think of $k_d\otimes_k M(R)$ as the following representation of the quiver $\widetilde{\mathsf{D}}_4$:
\begin{center}
 \begin{picture}(100,100)
  \put(85,-3){\makebox(0,0){$  k_d^n$}}
\put(85,30){\makebox(0,0){$  k_d^n$}}
\put(85,62){\makebox(0,0){$  k_d^n$}}
\put(85,97){\makebox(0,0){$  k_d^n$}}
  \put(15,49){\makebox(0,0){$  k_d^{2n}$}}
  \put(40,83){\makebox(0,0){$\scriptstyle (\id_n|R)$}}
  \put(54,66){\makebox(0,0){$\scriptstyle\Si^3((\id_n|R))$}}
  \put(59,45){\makebox(0,0){$\scriptstyle\Si^2((\id_n|R))$}}
  \put(67,22){\makebox(0,0){$\scriptstyle\Si((\id_n|R))$}}
  \put(23,58){\vector(4,3){53}}
  \put(23,49){\vector(4,1){53}}
  \put(23,44){\vector(4,-1){53}}
  \put(23,38){\vector(4,-3){53}}
\end{picture}
\end{center}
In a similar way, we can think of $N_{\phi_n(R)}$ as
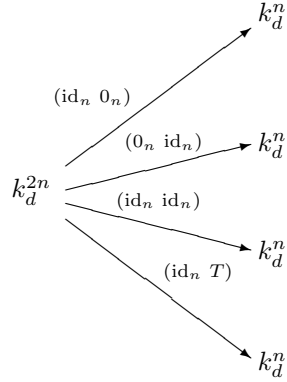
\begin{center}
\begin{picture}(90,135)
  \put(88,3){\makebox(0,0){$ k_d^n$,}}
\put(88,45){\makebox(0,0){$ k_d^n$}}
\put(88,86){\makebox(0,0){$ k_d^n$}}
\put(88,135){\makebox(0,0){$ k_d^n$}}
  \put(-3,69){\makebox(0,0){$ k_d^{2n}$}}
  \put(20,104){\makebox(0,0){$\scriptstyle(\id_n\,0_n)$}}
  \put(47,87){\makebox(0,0){$\scriptstyle(0_n\,\id_n)$}}
  \put(45,65){\makebox(0,0){$\scriptstyle(\id_n\,\id_n)$}}
  \put(60,38){\makebox(0,0){$\scriptstyle(\id_n\,T)$}}
  \put(10,78){\vector(4,3){70}}
  \put(10,69){\vector(4,1){70}}
  \put(10,64){\vector(4,-1){70}}
  \put(10,58){\vector(4,-3){70}}
  \end{picture}
\end{center}
with $T=\diag(\frac{1}{2}-i\phi_n(R)\E^\frac{1}{2},\,\frac{1}{2}-i\Si_n(\phi_n(R))\E^\frac{1}{2},\cdots,\,\frac{1}{2}-i\Si_n^{n-1}(\phi_n(R))\E^\frac{1}{2})$.

Let $A$ be defined as in the proof of Lemma \ref{lema2}, and $\phi\colon k_d\otimes_k M(R)\lra N_{\phi_n(R)}$ be defined by
 \begin{eqnarray}\label{isomorfismo}
\phi\coloneqq  \left(\begin{array}{c}(R-\Si^3(R))B\\(\Si^3(R)-R)A\\(\Si^2(R)-R)A\\(\Si(R)-\Si^3(R))B\\\left[\begin{array}{cc}-\Si^3(R)B&-R A\\B&A\end{array}\right]\end{array}\right),
 \end{eqnarray}
 where $B=(\Si^2(R)-\Si^3(R))^{-1}(\Si^2(R)-R)A$. We claim that $\phi$ is an isomorphism between $k_d\otimes_k M(R)$ and $N_{\phi_n(R)}$. By the definition of $R$ and $A$, all the matrices involved in the construction of $\phi$ are invertible. 
 
 Now, we just need to show that the corresponding diagrams commute. From our choices, we see that the first three diagrams commute:
 \begin{eqnarray*}
 [\id_n|R]\left[\begin{array}{cc}-\Si^3(R)B&-R A\\B&A\end{array}\right]&=&[(R-\Si^3(R))B][\id_n\,0_n],\\
 \Si^3([\id_n|R])\left[\begin{array}{cc}-\Si^3(R)B&-R A\\B&A\end{array}\right]&=&[(\Si^3(R)-R)A][0_n\,\id_n],\\
\Si^2([\id_n|R])\left[\begin{array}{cc}-\Si^3(R)B&-R A\\B&A\end{array}\right]&=&[(\Si^2(R)-R)A][\id_n\,\id_n].
 \end{eqnarray*} 
 For the last diagram, we have
 \begin{align*}
&B^{-1}(\Si(R)-\Si^3(R))^{-1}(\Si(R)-R)A\\
&=A^{-1}(\Si^2(R)-R)^{-1}(\Si^2(R)-\Si^3(R))(\Si(R)-\Si^3(R))^{-1}(\Si(R)-R)A&&\text{def. of }B\\
&=A^{-1} (\Si(R)-R)(\Si^3(R)-\Si^2(R))[(\Si^2(R)-R)(\Si^3(R)-\Si(R))]^{-1}A&&\text{comm. of } K[\A_n]\\
&=A^{-1}\left(\frac{1}{2}\id_n-i\E^\frac{1}{2}\left(\A(\nu_2^2)-\A(\nu_1)\A(\nu_3)\right)\left(\A(\nu_1^2)-\E \A(\nu_3^2)\right)^{-1}\right)A&&\text{ Equation (\ref{2.6})}\\
&=\frac{1}{2}\id_n-i\E^\frac{1}{2}A^{-1}\left(\left(\A(\nu_2^2)-\A(\nu_1)\A(\nu_3)\right)\left(\A(\nu_1^2)-\E \A(\nu_3^2)\right)^{-1}\right)A&&\text{distributivity}\\
&=\frac{1}{2}\id_n-i\E^\frac{1}{2}\diag(\phi_n(R),\Si_n(\phi_n(R)),\cdots,\,\Si_n^{n-1}(\phi_n(R))&&\text{Lemma \ref{lema2}}
 \end{align*}
Hence, $B^{-1}(\Si(R)-\Si^3(R))^{-1}(\Si(R)-R)A=T$, and thus, $k_d\otimes_k M(R)\cong N_{\phi_n(R)}$. Therefore, $M(R)$ is a simple regular representation of $\Lambda$.
\end{proof}
Now, let $M_{(1)}$ and $M_{(2)}$ be the representations whose defining matrices are $(1,\,\E^\frac{1}{2})$ and \\$(\id_2|\A_2\E^\frac{1}{4}-i\id_2\E^\frac{3}{4})$, respectively. With simple computations, one can check the following  result:
\begin{prop}\label{prop2.9}
Let $N_1\oplus\,^{\G}N_1\mbox{ and }N_2\oplus\,^{\G}N_2\oplus\,^{\G^2}N_2\oplus\,^{\G^3}N_2$ be representations as in \ref{prop2.2}. Then, 
$$N_1\oplus\,^{\G}N_1\cong K\otimes_kM_{(1)}\mbox{ and }$$
$$N_2\oplus\,^{\G}N_2\oplus\,^{\G^2}N_2\oplus\,^{\G^3}N_2\cong K\otimes_kM_{(2)}.$$
Moreover, $M_{(1)}$ and $M_{(2)}$ are simple regular representations.
\end{prop}
\begin{cor}\label{coro2.10}There is a 1-1 correspondence between  the isomorphism classes of the simple regular representations of $\Lambda$ and the isomorphism classes of the  ii-indecomposable regular representations of $\widetilde{\mathsf{D}}_4$.
\end{cor}
\begin{proof} Let $M$ be a simple regular  representations of $\Lambda$. From \ref{prop2.3}, for $X\cong k_d\otimes_k M$ there  exists an isomorphism $f\colon X\lra\,^{\G_d}X$ such that the automorphism $^{\G_d^{m-1}}f\circ\cdots\circ^{\G_d}f\circ f$ of $X$ is the identity. According to Section \ref{Invariantrep}, $X$ has to be isomorphic to one of the ii-indecomposable regular representations described in \ref{prop2.8} or \ref{prop2.9}. Conversely, for $X$ an ii-indecomposable regular representation as in \ref{prop2.8} and \ref{prop2.9}, we have that there exists $M$ a simple regular  representation of $\Lambda$ such that $k_d\otimes_k M\cong X$.\end{proof}

\begin{teo}\label{lema1}
Using the notation of Subsection \ref{mainresult}, we have that, for each $n\in \Z_{>0}$:
\begin{enumerate}
\item For $R\in K[\A_n]_{\reg}$, the representation $M(R)$ of $\Lambda$ is simple regular if and only if $\phi_n(R)\in k_n$ is generic. In this case, $\End_\Lambda(M(R))\cong k_n$. 

If $R,\,R'\in K[\A_n]_{\reg}$, then $M(R)\cong M(R')$ if and only if $\phi_n(R)=\Si_n^m(\phi_n(R'))$ for some $1\leq m\leq n-1$.

\item The representations $M_{(1)}$  and $M_{(2)}$ are simple regular with $\End_\Lambda(M_{(1)})\cong k_{2}$ and $\End_\Lambda(M_{(2)})\cong k_{4}$. 
\item Each simple regular representation of $\Lambda$ is isomorphic to  one of the types in  {\em a)} or {\em b)}.
\end{enumerate}
\end{teo}
\begin{proof}
\begin{enumerate}
\item The first part follows from Propositions \ref{prop2.3} and \ref{prop2.8}. 

The second part follows from the fact that, for any $R\in K[\A_n]_{\reg}$, we have the isomorphism $N_{\phi_n(R)}\cong\bigoplus_{j=0}^{n-1}M_{\frac{1}{2}-i\Si_n^j(\phi_n(R))\E^\frac{1}{2}}$. Consider  $R$ and $R'$ be in $K[\A_n]_{\reg}$. Then,  we have $M(R)\cong M(R')$ if and only if  $N_{\phi_n(R)}\cong N_{\phi_n(R')}$ if and only if $\bigoplus_{j=0}^{n-1}M_{\frac{1}{2}-i\Si_n^j(\phi_n(R))\E^\frac{1}{2}}\cong\bigoplus_{j=0}^{n-1}M_{\frac{1}{2}-i\Si_n^j(\phi_n(R'))\E^\frac{1}{2}}$ if and only if $\phi_n(R)=\Si_n^m(\phi_n(R'))$ for some $1\leq m\leq n-1$.

\item See Proposition \ref{prop2.9}. 

\item This follows from the Corollary \ref{coro2.10}.
\end{enumerate}
\end{proof}
\begin{remark}
It is an interesting fact that, for almost every simple regular representation $M$ with $\underline{\dim} M=(n,\,2n)$, we have $\End_\Lambda(M)\cong k_n$. The only exceptions are  when $M$ is isomorphic to  $M_{(1)}$ or $M_{(2)}$. In these cases, we have $\End_\Lambda(M_{(1)})\cong k_2$ and $\End_\Lambda(M_{(2)})\cong k_4$, respectively.
\end{remark}

\begin{remark}\label{remark2.3}{Using Theorems \ref{lema1} and \ref{teorema1.5}, we can conclude that  any simple regular representation is isomorphic to one of the following forms: $M_{(1)}$, $M_{(2)}$ or $M_n^{(j)}(s)$, where $[\id_n|R_n^{(j)}(s)]$ is the matrix of the last one, with $s\in k_n$ such that $\phi_n(R_n^{(j)}(s))$ (see Table \ref{table1}) is generic.}
\end{remark}
%%%%%%%%%%%%%%%%%%%%%%%%%%%%%%%% 
  \section{The canonical algebras of bimodule type \texorpdfstring{$(1,\,4)$}{(1,4)}.}\label{sec4}
  \subsection{Canonical algebras, Ringel's definition}\label{sec4.1}
 We review Ringel's definition \cite[Section 1]{Ri90} of canonical algebras over an arbitrary field for our setting.
 
  Recall that $k\coloneqq \C\parE$ is a quasi-finite field, see Section \ref{quasi1.2}. Thus, the finite dimensional division algebras we have to consider are all of the form $k_n=k[\E^\frac{1}{n}]$ for some $n\in\Z_{>0}$. We will only consider bimodules on which $k$ acts centrally. The category of these $k_a$-$k_b$ bimodules is equivalent to the category of $k_a\otimes_k k_b$ left modules.

For a $k_a$-$k_b$-bimodule $X$ and an automorphism $\Si_a$ of $k_a$, we define by $^{\Si_a}X$ the bi\-mo\-du\-le with left multiplication $y*x\coloneqq \Si_a(y)x$ for any $x\in X,\, y\in k_a$, and the right multiplication is the usual multiplication in $X_{k_b}$. One can define  similarly $X^{\Si_b}$ as a  $k_a$-$k_b$-bimodule.

In fact, from the above remarks, the category of $k$-$K$-bimodules with $k$ acting centrally is equivalent to the category of $k\otimes_k K$-modules.  

It is easy to see that $k_a\otimes_kk_b\cong k_{\mcm(a,\,b)}\times\cdots\times k_{\mcm(a,\,b)}\mbox{(${\mcd(a,\,b)}$ copies)}$ as a $k$-algebra. Thus, the category of $k_a$-$k_b$-bimodules is semisimple with $\mcd(a,\,b)$ isoclasses of simple objects.
  
Let  $a,\, b\in\N$ and $X=\,_{a}X_{b}$ be a $k_a$-$k_b$-bimodule such that $\dim\,_{k_a}X\dim X_{k_b}=4$. Then, it is easy to see that, in our situation, only two cases can occur: 
\begin{enumerate}
\item $k_a=K,\,k_b =k$ and  $X=K$.  
\item $k_a=k_b$ and $X=k_a\oplus \,^{\Si_a}k_a$, for some $a\in \N$. 
\end{enumerate}

In this paper, we consider only the first case.

Consider the set of isomorphism classes of simple regular representations of $X$, we denote by $\Phi$. Let $T\colon \Phi\lra \Z_{>0}$ be a function with $T(M)=1$ for almost all $M\in \Phi$.

Let $M_{1},\,M_{2},\cdots,\,M_{t}$ be the pairwise non-isomorphic elements of $\Phi$ with $l_i\coloneqq T(M_i)$ greater than $1$. We consider $M_i$ as a representation of $X$, with $\underline{\dim}M_i=(n_i,\,2n_i)$, e.g. as a $K$-linear map $\varphi_i\colon  K\otimes_{k} k^{2n_i}\lra K^{n_i}$. Let $k_{m_i}$ be the endomorphism ring of the representation $M_i$. Recall that $m_i=n_i$ when $M_i$ is homogeneous and $m_i=2n_i$ when $M_i$ is not homogeneous. Then, $k^{2n_i}$ is a $k$-$k_{m_i}^{\op}$ bimodule, and $K^{n_i}$ is a $K$-$k_{m_i}^{\op}$ bimodule. Thus, $\varphi_i$ is $K$-$k_{m_i}^{\op}$-linear. Let $M_i^+\colon V_i^{+}\lra K\otimes_{k} k^{2n_i}$ be the kernel of this map $\varphi_i$. Then, $V_i^+$ is, again, a $K$-$k_{m_i}^{\op}$-bimodule, and $M_i^+$ is $K$-$k_{m_i}^{\op}$-linear. 

Since $k^{2n_i}$ is $k$-$k_{m_i}^{\op}$-bimodule, $U_i^*\coloneqq \Hom_{k}(k^{2n_i},k)$ is a $k_{m_i}^{\op}$-$k$-bimodule, and we can consider the adjoint map $\widetilde{M_i}\colon  V_i^+\otimes_{k_{m_i}}U_i^*\lra K$ of $M_i^{+}$. 

We will see in the next section that $\widetilde{M}$ is surjective for any $M\in \Phi$. Following Ringel, we consider the species $\mathcal{S}=\mathcal{S}(T)$ which is displayed in terms of a decorated quiver in Figure \ref{figure1}.
\begin{figure}[ht]
\begin{center}
\begin{tikzpicture}
[->,>=stealth',shorten >=1pt,auto,node distance=2.3cm,thick,main node/.style=]
  \node[main node] (1) {$1_1$};   
  \node[main node] (5) [right  of=1]{$1_2$};   
 \node[main node] (6) [right  of=5]{$\cdots$};   
 \node[main node] (7) [right  of=6]{$1_{l_1-1}$};   
\node[main node] (2) [below  of=1] {$2_1$};
  \node[main node] (8) [right  of=2]{$2_2$};  
 \node[main node] (9) [right  of=8]{$\cdots$};   
 \node[main node] (10) [right  of=9]{$2_{l_2-1}$};   
 \node[main node] (11) [below right  of=10]{$k$};   
\node[main node] (3) [below left of=2] {$K$};
  \node[main node] (4) [below right of=3] {$t_1$};
  \node[main node] (12) [right  of=4]{$t_2$};   
 \node[main node] (13) [right  of=12]{$\cdots$};   
 \node[main node] (14) [right  of=13]{$t_{l_t-1}$}; 
 \node[main node](15)[right of=3]{$\vdots\,\,\,\,\,\,\,\,\,\,\,\,\,\,\,\,\,\,\,\,$};    
 \node[main node](16)[right of=15]{$\vdots\,\,\,\,\,\,\,\,\,\,\,\,\,\,\,\,\,\,\,\,$};    
 \node[main node](17)[right of=16]{};    
 \node[main node](18)[right of=17]{$\vdots\,\,\,\,\,\,\,\,\,\,\,\,\,\,\,\,\,\,\,\,$};    
\path[every node/.style={font=\sffamily\small}]     
(1) edge node   [left]{$V_1^+$} (3)        
 (5) edge node         {$k_{m_1}$} (1)             
 (6) edge node         {$k_{m_1}$} (5)             
(7) edge node         {$k_{m_1}$} (6)             
 (8) edge node         {$k_{m_2}$} (2)             
 (9) edge node         {$k_{m_2}$} (8)             
(10) edge node         {$k_{m_2}$} (9)             
 (12) edge node         {$k_{m_t}$} (4)             
 (13) edge node         {$k_{m_t}$} (12)             
(14) edge node         {$k_{m_t}$} (13)             
(2) edge node         {$V_2^{+}$} (3)             
  (4) edge node  {$V_t^{+}$} (3) 
  (11)edge node [right]{$U_1^*$}(7)
  edge node {$U_2^{*}$}(10)
  edge node {$U_t^{*}$}(14); 
\end{tikzpicture}
\end{center}
\caption{Modulated quiver of a canonical algebra of type $(1,\,4)$, where the field of the vertex $i_j$ is $\End (M_i)$.}\label{figure1}
\end{figure}
Let $\mathcal{T}$ be the tensor algebra of $\mathcal{S}$, and $\mathcal{R}$ be the ideal of $\mathcal{T}$ which is generated by the elements of the form 
$$v\otimes\underbrace{1\otimes\cdots\otimes1}_{l_i-2}\otimes u-\sum_rv_r\otimes\underbrace{1\otimes\cdots\otimes1}_{l_1-2}\otimes u_r,$$ 
  with $v\in V_i^+$, $u\in U_i^*$, all $v_r\in  V_1^+$, all $u_r\in U_1^*$, and such that $\widetilde{M_i}(v\otimes u)= \sum_r\widetilde{M_1}(v_r\otimes u_r)$. (These are elements in $(V_i^+\otimes k_{m_i}\otimes\cdots\otimes k_{m_i}\otimes U_i^*)\oplus(V_1^+\otimes k_{m_1}\otimes\cdots\otimes k_{m_1}\otimes U_1^*)$).
  
The {\bf canonical algebra of type} $T$ (over $K$) is, by definition, $C=C(T)=\mathcal{T}/\mathcal{R}$.
\subsection{The adjoint map}
Now, we want to describe the canonical algebras associated to the species of type $(1,\,4)$ over the field $k=\C\parE$. According to Ringel \cite{Ri90}, the main task for this end is to describe, for a map $\varphi_M\in \Hom(K\otimes U, V)$ of a regular representation $M$  with endomorphism ring $D$, the adjoint map $\widetilde{M}\colon  V^+\otimes_D U^*\lra K$, where $U^*\coloneqq \Hom_k(U,k)$ and $M^+\colon  V^+\lra K\otimes U$ is the kernel of $M$. Recall that in our situation, $D=k_m$ for some $m\in\N$, and $U=k^{2n}$ and $V=K^n$ for some $n$,  see Section \ref{sec1.3}. As we mentioned  in the last section, by construction, $V^+$ is a $K$-$D^{\op}$-bimodule, and $U^*$ is a $D^{\op}$-$k$-bimodule. Since $D\cong k_m$ and this is a commutative ring, we have $D^{\op}=D$.  As we said in Remark \ref{remark2.3}, any simple regular representation is isomorphic to one of the following forms: 
\begin{enumerate}
\item $M_{(1)}$, whose matrix is $(1,\E^\frac{1}{2})$.
\item $M_{(2)}$, whose matrix is $(\id_2|\A_2\E^\frac{1}{4}-i\id_2\E^\frac{3}{4})$.
\item $M_n^{(j)}(s)$, whose matrix is $[\id_n|R_n^{(j)}(s)]$, and is such  that $\phi_n[R_n^{(j)}(s)]=-i\frac{\nu_2^2-\nu_1\nu_3}{\nu_1^2-\nu_3^2\E}$ is generic, where  $R_n^{(j)}(s)=\A(\nu_1)\E^\frac{1}{4}+\A(\nu_2)\E^\frac{1}{2}+\A(\nu_3)\E^\frac{3}{4}$ is as in Table \ref{table1}, with $s\in k_n$. 
\end{enumerate}
Then, we have three cases, and one of them has five subcases. 
\begin{table}[ht]
\centering
\scalebox{0.95}{\begin{tabular}{|c|c|c|c|c|c|}\hline
$M$&$n$  & $D$ &$U^*$& $V^+$&$B$ a basis of $V^+\otimes_DU^*$\\ \hline
$M_{(1)}$&$1$&$k_2$& $k_2$&$K^{\Si}$ &$\{\B'_0\otimes e_1\}$\\ \hline
$M_{(2)}$&$2$&$K$& $K$&$K^{\Si}\times K^{\Si^2}$ &$\{\B'_0\otimes e_1,\,\B'_1\otimes e_1\}$\\ \hline
$M_n^{(r)}(s)$&$n\equiv1(\modu2)$&$k_n$& $k_n\times k_n$&$k_{4n}$&$\{\B'_j\otimes e_i\}_{0\leq j\leq n-1,\atop  i\in\{1,\,2\}}$ \\ \hline
$M_n^{(r)}(s)$&$n\equiv2(\modu4)$&$k_n$& $k_n\times k_n$&$k_{2n}\times k_{2n}$ &$\{\B'_{j,\,l}\otimes e_i\}_{l\in\{0,1\},\, i\in\{1,\,2\}, \atop 0\leq j\leq \frac{n-2}{2}}$\\ \hline
$M_n^{(r)}(s)$&$n\equiv0(\modu4)$&$k_n$& $k_n\times k_n$&$k_{n}\times k_{n}\times k_{n}\times k_{n}$&$\{\B'_{j,\,l}\otimes e_i\}_{{0\leq l\leq 3,\, i\in\{1,\,2\},\atop0\leq j\leq \frac{n-4}{4}}}$\\\hline
\end{tabular}}
\caption{Setup for the definition of the adjoint map of a simple regular representation in Proposition \ref{5.1}. We denote $1$, in the case of $M\cong M_{(1)}$ or $M_{(2)}$. Abusing notation by $e_1$.}
\label{table2}
\end{table}
\begin{prop}\label{5.1}
Let $M$ be a simple regular representation of $\Lambda$. Then, the following cases cover all possibilities for the adjoint  map.
\begin{enumerate}
\item Let $M=M_{(1)}$. Then, the adjoint map $\widetilde{M}_{(1)}$  is defined by $\B'_0\otimes 1\mapsto\E^\frac{1}{2}$, where $$\B'_0\coloneqq \E^\frac{1}{2}\otimes\tilde{e}_1-1\otimes\tilde{e}_2.$$ 
\item Let $M=M_{(2)}$. Then, the adjoint map $\widetilde{M}_{(2)}$ is defined by 
$$\B'_0\otimes 1\mapsto\E^\frac{-1}{2}\mbox{ and }\B'_1\otimes 1\mapsto-\E^\frac{-1}{2},$$
 where $\B'_j\coloneqq \B_0+(i)^{j+1}\E^\frac{1}{4}\B_1$, with $\B_0\coloneqq -i\E^\frac{-1}{2}\otimes\tilde{e}_1+(1+i)\E^\frac{-1}{4}\otimes\tilde{e}_2-1\otimes\tilde{e}_3\mbox{ and }$
$\B_1\coloneqq (1-i)\E^\frac{-3}{4}\otimes\tilde{e}_1+i\E^\frac{-1}{2}\otimes\tilde{e}_2-1\otimes\tilde{e}_4\mbox{.}$
\item Let $M=M_n^{(j)}(s)$,  whose matrix is $[\id_n|R_n^{(j)}(s)]$, with $R^{(j)}_n(s)$ as in Table \ref{table1}, for some $s=\sum_{j=0}^{n-1}s_j\E^\frac{j}{n}\in k_n$. In the following five  subcases,  we will consider $\D_{ij}$ as the Kronecker delta, $\B_0\coloneqq \sum_{i=0}^{n}(\sum_{j=1}^3\nu_{j\,n-i}\E^\frac{j}{4})\otimes\tilde{e}_i-1\otimes\tilde{e}_{2n}$ and  $\B_l=\B_0*\E^\frac{l}{n}$ for $1\leq l\leq n-1$. The adjoint map of $M$ for these cases appears in Table \ref{table3}.
 \end{enumerate}
\end{prop}
 \begin{sidewaystable}
 \centering
\begin{tabular}{|c|c|c|}\hline
$M$ & $\widetilde{M}$&$\B'_{j,\,l}$'s\\ \hline
& &\\
$\begin{array}{c}M_n^{(1)}(s)\\\mbox{and}\\n\equiv1\,(\modu 2)\end{array}$  &$\begin{array}{ccl}\B'_{j,\,0}\otimes e_1&\mapsto&\left\{\begin{array}{ll}s_{\frac{n-3}{2}-j}\E^\frac{1}{2}&0\leq j\leq \frac{n-3}{2}\\
s_{3\frac{n-1}{2}-j}\E^\frac{3}{2}+\delta_{n-1,\,j}e^\frac{\pi i}{4}\E^\frac{3}{4}&\frac{n+1}{2}\leq j \leq n-1.\end{array}\right.\\
\B'_{j,\,0}\otimes e_2&\mapsto& -\delta_{n-1,\,j}\end{array}$& $\B'_{j,\,0}=\B'_j$\\ &&\\ \hline&&\\
$\begin{array}{c}M_n^{(2)}(s)\\\mbox{and}\\n\equiv1\,(\modu 2)\end{array}$&$\begin{array}{ccl}\B'_{j,\,0}\otimes e_1&\mapsto &\delta_{0,\,n-(j+1)}e^\frac{3\pi i}{4}\E^\frac{1}{4}+s_{n-(j+1)}\E^\frac{1}{2}\\\B'_{j,\,0}\otimes e_2&\mapsto& -\delta_{n-1,\,j}\end{array}$&$\B'_{j,\,0}=\B'_j$ \\ &&\\ \hline&&\\
$\begin{array}{c}M_n^{(1)}(s)\\\mbox{and}\\n\equiv2\,(\modu 4)\end{array}$ &$\begin{array}{ccl}
\B'_{j,\,l}\otimes e_1&\mapsto &\frac{\left(s_{n-(j+1)}+(-1)^l\delta_{\frac{n-2}{2},\,j}\right)\E^\frac{1}{4}+(-1)^ls_{\frac{n}{2}-(j+1)}\frac{\E^\frac{3}{4}}{\E}}{2}\\
\B'_{j,\,l}\otimes e_2&\mapsto&\frac{(-\delta_{\frac{n-2}{2}\,j})^{l+1}\E^\frac{-1}{2}}{2}
\end{array}$&$\begin{array}{c}\frac{1}{2}\left(\B_j+(-1)^l\E^\frac{-1}{2}\B_{\frac{n+2j}{2}}\right)\\ 0\leq j\leq \frac{n-2}{2},\, l\in\{0,1\}\end{array}$
\\&&\\ \hline&&\\
$\begin{array}{c}M_n^{(2)}(s)\\\mbox{and}\\n\equiv2\,(\modu 4)\end{array}$  &$\begin{array}{ccl}
\B'_{j,\,l}\otimes e_1&\mapsto &\left\{\begin{array}{ll}\frac{(1-i)\left(s_{\frac{n}{2}-(2+j)}\E^\frac{1}{2}+(-1)^ls_{n-(2+j)}\E\right)}{2}&0\leq j\leq \frac{n-4}{2}\\\frac{(1-i)(s_{n-1}\E^\frac{3}{2}+(-1)^ls_\frac{n-2}{2}\E)+(i+(-1)^l)\E^\frac{1}{4}}{2}&j=\frac{n-2}{2}
\end{array}\right.\\
\B'_{j,\,l}\otimes e_2&\mapsto&\frac{(-\delta_{\frac{n-2}{2}\,j})^{l+1}\E^\frac{-1}{2}}{2}
\end{array}$&$\begin{array}{c}\frac{1}{2}\left(\B_j+(-1)^l\E^\frac{-1}{2}\B_{\frac{n+2j}{2}}\right)\\ 0\leq j\leq \frac{n-2}{2},\, l\in\{0,1\}\end{array}$\\& &\\ \hline &&\\
$\begin{array}{c}M_n^{(1)}(s)\\\mbox{and}\\n\equiv0\,(\modu 4)\end{array}$&$\begin{array}{ccl}
\B'_{j,\,l}\otimes e_1&\mapsto&\frac{1}{4}\left[\left(\sum_{r=0}^3i^{lr}s_{\frac{n(4-r)}{4}-(j+1)}\E^\frac{1-r}{4}\right)+\D_{\frac{n-4}{4},\,j}i^{2l}(1-i)\right]\\
\B'_{j,\,l}\otimes e_2&\mapsto&\frac{i^{6-l}\delta_{\frac{n-4}{4}\,j}\E^\frac{-3}{4}}{4}\end{array}$&$\begin{array}{c}\frac{1}{4}\left(\B_j+i^l\E^\frac{-1}{4}\B_{\frac{n}{4}+j}+i^{2l}\E^\frac{-1}{2}\B_{\frac{n}{2}+j}+i^{3l}\E^\frac{-3}{4}\B_{\frac{3n}{4}+j}\right)\\ 0\leq j\leq \frac{n-4}{4},\,0\leq l\leq3\end{array}$\\& &\\ \hline
\end{tabular}
\caption{Adjoint map of a simple regular homogeneous representation. Here, $\D_{i,\,j}$ is the Kronecker delta.}\label{table3}
\end{sidewaystable}
\begin{proof}
We will verify our formula case by case, according to the classification of the simple regular representations in Theorem \ref{lema1}:
\begin{enumerate}
\item  Let $\varphi_{M_{(1)}}\colon K\otimes k^2\lra K$ be defined by $1\otimes \tilde{e}_1\mapsto1$ and $1\otimes \tilde{e}_2\mapsto\E^\frac{1}{2}$. We have 
$$D=\left\{\left(\left[\begin{array}{cc}a&\E b\\b&a\end{array}\right],\,a+b\E^\frac{1}{2}\right)|a,b\in k\right\}\cong k_2.$$  

  It is easy to see that $U^*\cong k_2$ as $D$-$k$-bimodule, with $e_1\mapsto\E^\frac{1}{2}$ and $e_2\mapsto 1$. 

Clearly,  $V^+$ is generated by $\B'_0\coloneqq \E^\frac{1}{2}\otimes\tilde{e}_1-1\otimes\tilde{e}_2$, as a $K$-module. Note that if $\B'_0*(a+b\E^\frac{1}{2})=(a-b\E^\frac{1}{2})\B'_0$,  then $V^+\cong K^\Si$.  One can check that $B\coloneqq \{\B'_0\otimes1\}$ is a $K$-basis of $V^+\otimes_D U^*$. Thus, $\widetilde{M}_{{(1)}}\colon V^+\otimes_D U^*\lra K$  is the adjoint map defined by $\B'_0\otimes 1\mapsto -1$. 
\item Let $\varphi_{M_{(2)}}\colon K\otimes k^4\lra K^2$ be the simple regular representation defined as above. One can see that $M_{(2)}$ is isomorphic to the representation whose map is defined by $1\otimes \tilde{e}_1\mapsto e_1$, $1\otimes \tilde{e}_2\mapsto e_2$, $1\otimes \tilde{e}_3\mapsto(-i\E^\frac{-1}{2},(1+i)\E^\frac{-1}{4})$ and $1\otimes \tilde{e}_4\mapsto((1-i)\E^\frac{-3}{4},i\E^\frac{-1}{2})$. We have that $D$ is
$$\left\{\left(\sum_{j=0}^3x_j\A_4^j,\,\left[\begin{array}{cc}x_0+(1-i)x_1\E^\frac{1}{4}-ix_2\E^\frac{1}{2}&x_1+(1-i)x_2\E^\frac{1}{4}-ix_3\E^\frac{1}{2}\\ \E x_3+ix_1\E^\frac{1}{2}+(1+i)x_2\E^\frac{3}{4}&x_0+ix_2\E^\frac{1}{2}+(1+i)x_3\E^\frac{3}{4}\end{array}\right]\right)| x_j\in k\right\}.$$  
Thus,  $D\cong K$. Then, $U^*\cong K$ as a $D$-$k$-bimodule. Let us define  
$$\B_0\coloneqq -i\E^\frac{-1}{2}\otimes\tilde{e}_1+(1+i)\E^\frac{-1}{4}\otimes\tilde{e}_2-1\otimes\tilde{e}_3\mbox{ and}$$
$$\B_1\coloneqq (1-i)\E^\frac{-3}{4}\otimes\tilde{e}_1+i\E^\frac{-1}{2}\otimes\tilde{e}_2-1\otimes\tilde{e}_4\mbox{.}$$
Then, $V^+$ is generated as a $K$-module by  $\B'_0\coloneqq \B_0+i\E^\frac{1}{4}\B_1$ and $\B'_1\coloneqq \B_0-\E^\frac{1}{4}\B_1$. Since $\B'_j*\E^\frac{1}{4}=i^{j+1}\E^\frac{1}{4}*\B'_j$, we have $V^+\cong K^\Si\times K^{\Si^2}$.  It is easy to see that $B\coloneqq \{\B'_0\otimes1,\,\B'_1\otimes1\}$  is a $K$-basis of $V^+\otimes_D U^*$. 
 
 Thus, $\widetilde{M}_{(2)}\colon V^+\otimes_D U^*\lra K$ is the adjoint map, defined by $\B'_0\otimes 1\mapsto\E^\frac{-1}{2}$ and $\B'_1\otimes 1\mapsto-\E^\frac{-1}{2}$.
 
 \item Let $M^{(j)}_n(s)$ be the representation whose matrix is $(\id_n|R^{(j)}_n(s))$, where $\phi_n(R^{(j)}_n(s))$ is generic. Recall that  $R^{(j)}_n(s)$ is equal  to $\A(\nu_1)\E^\frac{1}{4}+\A(\nu_2)\E^\frac{1}{2}+\A(\nu_3)\E^\frac{3}{4}$ as in Table \ref{table1}, with $s=\sum_{j=0}^{n-1}s_j\E^\frac{j}{n}\in k_n$.  Since $D\cong k_{n}$, we have  $U^*\cong k_{n}\times k_{n}$ as $D$-$k$-bimodule. 
 
 Let $\nu_i=\sum_{j=0}^{n-1}\nu_{ij}\E^\frac{j}{n}$ be in $k_n$. Next, we define $a_j\coloneqq \nu_{1\, j}\E^\frac{1}{4}+\nu_{2\,j}\E^\frac{1}{2}+\nu_{3\,j}\E^\frac{3}{4}\in K$, for $0\leq j\leq n-1$, and $\B_l\coloneqq \B_0*\E^\frac{l}{n}$, for $1\leq l\leq n-1$, where 
$$\B_0\coloneqq \sum_{i=1}^{n}a_{n-i}\otimes\tilde{e}_i-1\otimes\tilde{e}_{2n}.$$
Now, we have to distinguish several subcases according to the parity of $n$ modulo $4$.
\begin{enumerate}
\item Let $n\equiv 1(\modu2)$. One can see that $V^+$ is generated by $\{ \B_0,\,\B_1,\cdots,\B_{n-1}\}$ as a $K$-module. Moreover, $V^+\cong k_{4n}$. Now, to obtain the same notation, we will denote $\B'_{j,\,0}$ as $\B_j$.

 Calculation shows that $B=\{\B'_{j,\,0}\otimes e_i|0\leq j\leq n-1, i\in\{1,\,2\}\}$ is a $K$-basis of $V^+\otimes_D U^*$. So $\widetilde{M}\colon V^+\otimes_D U^*\lra K$ is defined by 
\begin{eqnarray*}
\B'_{j,\,0}\otimes e_1&\mapsto&a_{n-(j+1)}\\
\B'_{j,\,0}\otimes e_2&\mapsto&-\delta_{n-1\,j}.
\end{eqnarray*}
More precisely, we have that:
\begin{itemize}
\item If $M=M_n^{(1)}(s)$, then 
$$\B'_{j,\,0}\otimes e_1\mapsto\left\{\begin{array}{ll}s_{\frac{n-3}{2}-j}\E^\frac{1}{2}&0\leq j\leq \frac{n-3}{2}\\
s_{3\frac{n-1}{2}-j}\E^\frac{3}{2}+\delta_{n-1\,j}e^\frac{\pi i}{4}\E^\frac{3}{4}&\frac{n-1}{2}\leq j \leq n-1.\end{array}\right.$$
\item If $M=M_n^{(2)}(s)$, then $\B'_{j,\,0}\otimes e_1\mapsto \delta_{0\,n-(j+1)}e^\frac{3\pi i}{4}\E^\frac{1}{4}+s_{n-(j+1)}\E^\frac{1}{2}$.
\end{itemize}
 \item Let $n\equiv2 (\modu 4)$. For $0\leq j\leq\frac{n-2}{2}$ and $l\in\{0,\,1\}$, we define the element $\B'_{j,\,l}\coloneqq \frac{1}{2}\left(\B_j+(-1)^l\E^\frac{-1}{2}\B_{\frac{n+2j}{2}}\right)$ in $V^+$.  Since 
 $$\B_j=\left\{\begin{array}{cl}\B'_{j,\,0}+\B'_{j,\,1}&\mbox{ for  }0\leq j\leq \frac{n-2}{2}\\\E^\frac{1}{2}(\B'_{j-\frac{n}{2},\,0}-\B'_{j-\frac{n}{2},\,1})&\mbox{ for  }\frac{n}{2}\leq j\leq n-1,\end{array}\right.$$
 we have that $V^+$ is generated by $\{\B'_{j,l}|0\leq j\leq \frac{n-2}{2},\,l\in\{0,\,1\}\}$ as a $K$-module. Now, as in Lemma \ref{2.3.1}, we have $V^+\cong k_{2{n}}\times k_{2{n}}$ as a $K$-module, where $ \B'_{j,\,l}\mapsto \E^\frac{j}{n}e_{l+1}$. Moreover, from our definition of $\B_j$'s, we have $\B'_{j,\,l}*\E^\frac{1}{n}=\B'_{j+1,\,l}$ for $0\leq j\leq \frac{n-4}{2}$ and $\B'_{\frac{n-2}{2},\,l} *\E^\frac{1}{n}=\E\B'_{0,l}$. Hence, $V^+\cong k_{2{n}}\times k_{2{n}}$ as $K$-$k_n$-bimodules.
 
 It is easy to check that  $B\coloneqq \{\B'_{j,\,l}\otimes e_i|0\leq j\leq \frac{n-2}{2},\, l\in\{0,1\},\, i\in\{1,\,2\}\}$ is  a $K$-basis of $V^+\otimes_D U^*$. Thus, $\widetilde{M}\colon V^+\otimes_D U^*\lra K$ defined by 
\begin{eqnarray*}
\B'_{j,\,l}\otimes e_1&\mapsto&\frac{a_{n-(j+1)}+(-1)^la_{\frac{n}{2}-(j+1)}\E^\frac{-1}{2}}{2}
\\
\B'_{j,\,l}\otimes e_2&\mapsto&\frac{(-\delta_{\frac{n-2}{2}\,j})^{l+1}\E^\frac{-1}{2}}{2}.
\end{eqnarray*}is the adjoint map. More precisely, we have: 
\begin{itemize}
\item If $M=M_n^{(1)}(s)$, then
$$\B'_{j,\,l}\otimes e_1\mapsto \frac{\left(s_{n-(j+1)}+(-1)^l\delta_{\frac{n-2}{2},\,j}\right)\E^\frac{1}{4}+(-1)^ls_{\frac{n}{2}-(j+1)}\frac{\E^\frac{3}{4}}{\E}}{2}.$$
\item If $M=M_n^{(2)}(s)$, then 
$$\B'_{j,\,l}\otimes e_1\mapsto\left\{\begin{array}{ll}\frac{(1-i)\left(s_{\frac{n}{2}-(2+j)}\E^\frac{1}{2}+(-1)^ls_{n-(2+j)}\E\right)}{2}&0\leq j\leq \frac{n-4}{2},\\\frac{(1-i)(s_{n-1}\E^\frac{3}{2}+(-1)^ls_\frac{n-2}{2}\E)+(i+(-1)^l)\E^\frac{1}{4}}{2}&j=\frac{n-2}{2}.
\end{array}\right.$$
\end{itemize}
\item Let $n\equiv0(\modu 4)$. For  $0\leq j \leq \frac{n-4}{4}$,
$$\B'_{j,l\,}\coloneqq \frac{1}{4}\left(\B_j+i^l\E^\frac{-1}{4}\B_{\frac{n}{4}+j}+i^{2l}\E^\frac{-1}{2}\B_{\frac{n}{2}+j}+i^{3l}\E^\frac{-3}{4}\B_{\frac{3n}{4}+j}\right),$$
 with $0\leq l\leq 3$. For $\frac{rn}{4}\leq j\leq \frac{(r+1)n}{4}-1$ with $0\leq r\leq 3$, it is easy to see  that
 $$\B_j=\E^\frac{r}{4}\sum_{l=0}^3i^{(4-l)r}\B'_{j-\frac{rn}{4},\,l}.$$ 
 Hence,  $V^+$ is generated by $\{\B'_{j,\,l}|0\leq j\leq \frac{n-4}{4},\,0\leq l\leq3\}$, as a $K$-module. It follows that $V^+\cong k_{n}\times k_{n}\times k_{n}\times k_{n}$, where $\B'_{j,\,l}\mapsto \E^\frac{j}{n}e_{l+1}$. Moreover, from our definition of $\B_j$'s, we have that the isomorphism is of  $K$-$k_n$-bimodules.

One can verify that $B\coloneqq \{\B'_{j,\,l}\otimes e_i|0\leq j\leq \frac{n-4}{4}, 0\leq l\leq 3,\, i\in\{1,\,2\}\}$ is a $K$-basis of $V^+\otimes_D U^*$. In this case the unique option for $M$ is $M_n^{(1)}(s)$. Thus, $\widetilde{M}\colon V^+\otimes_D U^*\lra K$, defined by 
\begin{eqnarray*}
\B'_{j,\,l}\otimes e_1&\mapsto&\frac{1}{4}\left[\left(\sum_{r=0}^3i^{lr}s_{\frac{n(4-r)}{4}-(j+1)}\E^\frac{1-r}{4}\right)+\D_{\frac{n-4}{4},\,j}i^{2l}(1-i)\right]\\
\B'_{j,\,l}\otimes e_2&\mapsto&\frac{i^{6-l}\delta_{\frac{n-4}{4}\,j}\E^\frac{-3}{4}}{4}.
\end{eqnarray*}
 is the adjoint map in this case.
\end{enumerate}  
\end{enumerate}
\end{proof}

\begin{example}
{ Since $\phi[R_3^{(1)}(1)]=\E^\frac{1}{3}$, we have   $M_3^{(1)}(1)$ is a  simple regular representation. According to Table \ref{table3}, its adjoint map $\widetilde{M_3^{(1)}(1)}\colon k_{12}\otimes k_3^2\lra K$ is defined by 
$$1\otimes e_1\mapsto\E^\frac{1}{2},\,\E^\frac{1}{3}\otimes e_1\mapsto0,\,\E^\frac{2}{3}\otimes e_1\mapsto e^\frac{\pi i}{4}\E^\frac{3}{4},\,1\otimes e_2\mapsto0,\,\E^\frac{1}{3}\otimes e_2\mapsto0,\,\E^\frac{2}{3}\otimes e_2\mapsto -1.$$
}\end{example}
\subsection{The canonical algebras of \texorpdfstring{$\Lambda$}{Λ}}\label{3.3}
 Consider $t\geq 1$ and $l\coloneqq (l_1,\,l_2,\cdots,\,l_t)$ with $l_j\geq 2$. We consider  a family $M_1,\,M_2,\cdots,\,M_t$ of  simple regular representations that are pairwise non-isomorphic. Recall that $M_j$ is isomorphic to one of  $M_{(1)}$, $M_{(2)}$ or $M(R)$ (see Theorem \ref{lema1}). 
 
Following Ringel,  we take the species which is determined by the decorated quiver in Figure \ref{figure1}. Let $\mathcal{T}$  be the  tensor algebra of our species. For $x=v\otimes u\in B_j$, we define $\bar{x}\in\mathcal{T}$ as $\overline{x}\coloneqq  v\otimes\underbrace{1\otimes\cdots\otimes1}_{l_j-2}\otimes u$, where $B_j$ is the basis of $V_j^+\otimes_{D_j} U_j^*$  as in Table \ref{table2}.

 Without loss of generality, we can assume that there is an $x_0\in B_1$ such that $$\cogrado(\widetilde{M}_1(x_0))=\min\{\cogrado(\widetilde{M}_j(x))|x\in B_j,\,1\leq j\leq t\}.$$
 Recall that $\ima (\widetilde{M}_j)=K$, for all $1\leq j \leq t$. Let $\mathcal{R}$ be the ideal of $\mathcal{T}$ generated by the elements of the form
  $$v\otimes1\otimes\cdots\otimes1\otimes u-\sum_rv_r\otimes1\otimes\cdots\otimes1\otimes u_r,$$
  with $v\in V_j^+$, $u\in U_j^*$, $v_r\in V_1^+$, $u_r\in U_1^*$, and such that $\widetilde{M}_j(v\otimes u)= \sum_r\widetilde{M}_1(v_r\otimes u_r)$. We have two cases:
\begin{enumerate}
\item For $\cogrado(\widetilde{M}_1(x_0))\geq0$, we can assume that $\widetilde{M}_1(x_0)=-1$.   For all $x\in B_j$, we have  
  $$\widetilde{M}_1(-\widetilde{M}_j(x)*x_0)=-\widetilde{M}_j(x)*\widetilde{M}_1(x_0)=-\widetilde{M}_j(x)*-1=\widetilde{M}_j(x)\in\C[\![\E^\frac{1}{4}]\!].$$
 Since $B_j$ is a basis of $V_j^+\otimes_{D_j} U_j^*$, it follows that $\mathcal{R}$ is generated by:
  $$\{\overline{x}+\widetilde{M}_j(x)\overline{x_0}|x\in B_j, 1\leq j\leq t\}.$$
\item For $\cogrado(\widetilde{M}_1(x_0))<0$, we have $\widetilde{M}_1(x_0)^{-1}\in \C[\![\E^\frac{1}{4}]\!]$. Moreover, for all $x\in B_j$, $1\leq j\leq t$, we have $\widetilde{M}_j(x)\widetilde{M}_1(x_0)^{-1}\in \C[\![\E^\frac{1}{4}]\!]$ . 

As above, we can see that $\widetilde{M}_1(\widetilde{M}_j(x)\widetilde{M}_1(x_0)^{-1}*x_0)=\widetilde{M}_j(x)$. Thus, $\mathcal{R}$ is generated by:
  $$\{\overline{x}-\widetilde{M}_j(x)\widetilde{M}_1(x_0)^{-1}\overline{x_0}|x\in B_j, 1\leq j\leq t\}.$$
  \end{enumerate}
  Therefore, $\mathcal{R}$ is completely defined by $\{M_j\}_{1\leq j\leq t}$. Then, we can denote the canonical algebra $\mathcal{T}/\mathcal{R}$ as $\mathcal{C}(l,\{M_j\}_{1\leq j\leq t})$.
  \begin{example}\label{ex3.3}{
Let $M_1\coloneqq M_{(1)}$ and $M_2\coloneqq M_3^{(1)}(1)$ be simple regular representations and $l=(2,\,3)$. Then, the modulated quiver is represented in Figure \ref{quiver3.3} and the  generators of $\mathcal{R}$ are
$$\begin{array}{ccccc}
1\otimes1\otimes e_1+\E^\frac{1}{2}(1\otimes1),&\quad&\E^\frac{2}{3}\otimes1\otimes e_1+e^\frac{\pi i}{4}(1\otimes1),&\quad&\E^\frac{2}{3}\otimes1\otimes e_2-(1\otimes1),\\1\otimes1\otimes e_2,&&\E^\frac{1}{3}\otimes1\otimes e_1,&&\E^\frac{1}{3}\otimes1\otimes e_2.\end{array}$$
\begin{figure}[ht]
\begin{center}
\begin{tikzpicture}
[->,>=stealth',shorten >=1pt,auto,node distance=1.5cm,thick,main node/.style=]
  \node[main node] (1) {$k_2$};   
 \node[main node] (5) [below left of=1] {};
  \node[main node] (3) [left of=5] {$K$};
  \node[main node] (4) [below right of=3] {$k_3$};
  \node[main node] (6) [ right of=4] {};
  \node[main node] (12) [right  of=6]{$k_3$};   
\node[main node] (11) [above right  of=12]{$k$};   
\path[every node/.style={font=\sffamily\small}]     
(1) edge node   {$K^{\Si}$} (3)        
 (11) edge node         {$k_{2}$} (1)             
  (12) edge node         {$k_{3}$} (4)             
 (11) edge [bend right=10] node  [left]{$k_3$} (12)        
(11) edge [bend left=10] node     {$k_3$} (12)
 (4) edge node  {$k_{12}$} (3); 
\end{tikzpicture}
\end{center}
\caption{Modulated quiver of Example \ref{ex3.3}.}
\label{quiver3.3}
\end{figure}\\

It is easy to see that $\mathcal{C}((2,\,3),\{M_1,M_2\})$ is isomorphic to 
$$\left[\begin{array}{ccccc}K&K^\Si&k_{12}&k_{12}&K\\
0&k_2&0&0&k_2\\
0&0&k_3&k_3&k_3^2\\
0&0&0&k_3&k_3^2\\
0&0&0&0&k
\end{array}
\right]$$
with the relations:
\begin{eqnarray*}
E_{12}E_{25}&=&-E_{15},\\
E_{13}e_1E_{35}=E_{14}e_1E_{45}&=&-\E^\frac{1}{2}E_{12}E_{25},\\
\E^\frac{2}{3}E_{13}e_1E_{35}=\E^\frac{2}{3}E_{14}e_1E_{45}&=&-e^\frac{\pi i}{4}\E^\frac{3}{4}E_{12}E_{25},\\
\E^\frac{2}{3}E_{13}e_2E_{35}=\E^\frac{2}{3}E_{14}e_2E_{45}&=&E_{12}E_{25},\\
E_{13}e_2E_{35}=E_{14}e_2E_{45}&=&0,\\
\E^\frac{1}{3}E_{13}e_1E_{35}=\E^\frac{1}{3}E_{14}e_1E_{45}&=&0,\\
\E^\frac{1}{3}E_{13}e_2E_{35}=\E^\frac{1}{3}E_{14}e_2E_{45}&=&0.\\
\end{eqnarray*}
}
\end{example}
  \subsection{The algebra \texorpdfstring{$\widetilde{H}^{\tw}(l, P)$}{Htw(l,P)}}\label{3.4}
  Consider $t\geq 1$, $l\coloneqq (l_1,\cdots,l_t)\in\N_{\geq 2}^t$ and $P\coloneqq\{\rho_1,\cdots,\,\rho_t\}\subset\spec(k[x])\dot\cup\{1,\,2\}$. Let $\rho_j$ be an element in  $P$. Then, $\rho_j$ is in $\spec(k[x])$ or $\{1,\,2\}$. In both cases, there is a simple regular representation $M_j$.
  
  Now, we can define symmetrizable Cartan matrices $C_{(l,\,P)}\in\M_{m\times m}(\Z)$ with the symmetrizer, $D_{(l,\,P)}=\diag(c_1,\cdots,c_m)$, where $m=\sum_{j=1}^tl_j-(t-2)$, that satisfy the following conditions:
  
   There are $I_1,\cdots,I_t\subseteq\{1,\,2,\cdots,m\}$ such that
  \begin{enumerate}
  \item  $|I_j|=l_j+1$, $\bigcup_{j=1}^tI_j=\{1,\,2,\cdots,m\}$ and $I_i\cap I_j=\{1,\, m\}$ for $i\neq j$.
  \item For each $j\in\{1,\cdots,\, t\}$, if $M_j\cong M_{(1)},\,M_{(2)}$ or  is a $\widetilde{\mathsf{D}}_4$-homogeneous simple regular representation ($M_n^{(j)}(s)$), then $C_{I_j}= C_{(1\,2)},\,C_{(2\,4)}$ or $C_{(n)}$ in $\M_{(l_j+1)\times(l_j+1)}(\Z)$, respectively, with  
$$C_{(1\,2)}=\left[\begin{array}{cccccccc}2&-1&0& & & & &\\-2&2&-1&& & &&\\0&-1&2 &\cdot & & &&\\ & &\cdot & \cdot& \cdot& &&\\ & & &\cdot &\cdot & \cdot&&\\ && &&\cdot &\cdot &-1&0\\ && && &-1&2&-1\\ && && &0 &-2&2\\\end{array}\right]\mbox{ and }D_{(1\,2)}=\diag(4,2,\cdots,2,1),$$ 
$$C_{(2\,4)}=\left[\begin{array}{cccccccc}2&-2&0& & & & &\\-2&2&-1&& & &&\\0&-1&2 &\cdot & & &&\\ & &\cdot & \cdot& \cdot& &&\\ & & &\cdot &\cdot & \cdot&&\\ && &&\cdot &\cdot &-1&0\\ && && &-1&2&-1\\ && && &0 &-4&2\\\end{array}\right]\mbox{ and }D_{(2\,4)}=\diag(4,4,\cdots,4,1),$$
$$C_{(n)}=\left[\begin{array}{cccccccc}2&-n&0& & & & &\\-4&2&-1&& & &&\\0&-1&2 &\cdot & & &&\\ & &\cdot & \cdot& \cdot& &&\\ & & &\cdot &\cdot & \cdot&&\\ && &&\cdot &\cdot &-1&0\\ && && &-1&2&-2\\ && && &0 &-2n&2\\\end{array}\right]\mbox{ and }D_{(n)}=\diag(4,n,\cdots,n,1).$$
  (Recall that $C_J$ is the matrix obtained from $C$ by deleting rows and columns and indexing by $J$).
\end{enumerate}  
 Now, let $\mathcal{I}\coloneqq \{I_1,\cdots,\,I_t\}$ be as above. We can see that $I_j=\{j_1,\,j_2,\cdots,\,j_{l_j+1}\}$, with $j_1=1$, $j_{l_j+1}=m$ and $j_i<j_{i+1}$ for all $i\in\{1,\cdots,\,l_j\}$ and $j\in\{1,\cdots,\,t\}$. We define $\Omega_{\mathcal{I}}$ as
  $$\Omega_{\mathcal{I}}=\{(j_i,\,j_{i+1})|1\leq j\leq t,\,1\leq i\leq l_j\}.$$
 
  \begin{remark} 
 { Let $C_{(l,\,P)}$, $D_{(l,\,P)}$, $\mathcal{I}$ and $\Omega_{\mathcal{I}}$ be as above. As in \cite{GLS16}, we can define the following relations 
 \begin{eqnarray*}
\begin{array}{ll} \epsilon^2_{a}\A_{ab}=-\A_{ab}\epsilon_{b}& \mbox{If $a=j_1$ and $C_{I_j}=C_{(1,\,2)}$,}\\ &\\\epsilon_{a}\A_{ab}^{(g)}=(-i)^g\A_{ab}^{(g)}\epsilon_{b}& \mbox{if $a=j_1$ and $C_{I_j}=C_{(2,\,4)}$,}\\&\\  \epsilon^{f_{ba}}_{a}\A_{ab}^{(g)}=\A_{ab}^{(g)}\epsilon_{b}^{f_{ab}}&\mbox{otherwise, (where $f_{ab}=-c_{ab}/g_{ab}$),}\end{array}
\end{eqnarray*}
with  $g_{ab}\coloneqq\mcd(-c_{ab},-c_{ba})$ and $1\leq g\leq g_{ab}$.

Of these relations we can see, following  Table \ref{table4}, that every path from $m$ to $1$ is $\epsilon_{j_1}^r p^{(j)}_{g\,l\,i}$, $r\in\Z_{\geq0}$, where $p^{(j)}_{g\,l\,i}$ is equal to $\A_{j_1\,j_2}^{(g)}\epsilon_{j_2}^l\A_{j_2\,j_3}\cdots\A_{j_{l_j-1}\,j_{l_j}}\A^{(i)}_{j_{l_j}\,j_{l_j+1}}$. 
 }\end{remark}  
 \begin{table}[ht]
\centering
\scalebox{0.85}{
\begin{tabular}{|c|c|c|c|c|}\hline
$C_{I_j}$& $j$-th path&$g$&$l$&$i$ \\ \hline
$C_{(1\,2)}$&\begin{tikzpicture}
[->,>=stealth',shorten >=1pt,auto,node distance=1.3cm,thick,main node/.style=]
  \node[main node] (1) {${j_1}$};
  \node[main node] (2) [right of=1]{$j_2$};
  \node[main node] (3) [right of=2]{$\cdots$};   
  \node[main node] (4) [right of=3] {$j_{l_j}$};  
  \node[main node] (5) [right of=4]{$j_{l_j+1}$};  
\path[every node/.style={font=\sffamily\small}]     
(1) edge [loop above] node              {} (1) 
(2) edge [loop above] node              {} (2) 
(2) edge  node  {} (1)        
(3) edge node                              {} (2)  
(4) edge [loop above] node              {} (4)
(4) edge node {}(3)
(5) edge [loop above] node              {} (5)
(5) edge node {}(4)
;
\end{tikzpicture}&$1$&$0$&$1$
\\ \hline
$C_{(2\,4)}$&\begin{tikzpicture}
[->,>=stealth',shorten >=1pt,auto,node distance=1.3cm,thick,main node/.style=]
 \node[main node] (1) {${j_1}$};
  \node[main node] (2) [right of=1]{$j_2$};
  \node[main node] (3) [right of=2]{$\cdots$};   
  \node[main node] (4) [right of=3] {$j_{l_j}$};  
  \node[main node] (5) [right of=4]{$j_{l_j+1}$};
\path[every node/.style={font=\sffamily\small}]     
(1) edge [loop above] node              {} (1) 
(2) edge [loop above] node              {} (2) 
(2) edge [bend left=20] node  {} (1)        
(2) edge [bend right=20] node              {} (1)
(3) edge node                              {} (2)  
(4) edge [loop above] node              {} (4)
(4) edge node {}(3)
(5) edge [loop above] node              {} (5)
(5) edge node  {} (4)        
;
\end{tikzpicture}&$1\leq g\leq 2$&$0$&$1$\\ \hline
$C_{(n)}$ and $n\equiv1(\modu2)$&\begin{tikzpicture}
[->,>=stealth',shorten >=1pt,auto,node distance=1.3cm,thick,main node/.style=]
  \node[main node] (1) {${j_1}$};
  \node[main node] (2) [right of=1]{$j_2$};
  \node[main node] (3) [right of=2]{$\cdots$};   
  \node[main node] (4) [right of=3] {$j_{l_j}$};  
  \node[main node] (5) [right of=4]{$j_{l_j+1}$};
\path[every node/.style={font=\sffamily\small}]     
(1) edge [loop above] node              {} (1) 
(2) edge [loop above] node              {} (2) 
(2) edge  node  {} (1)        
(3) edge node                              {} (2)  
(4) edge [loop above] node              {} (4)
(4) edge node {}(3)
(5) edge [loop above] node              {} (5)
(5) edge [bend right=20] node  {} (4)        
(5) edge [bend left=20] node              {} (4)
;
\end{tikzpicture}&$1\leq g\leq n$&$0$&$1\leq i\leq 2$\\ \hline
$C_{(n)}$ and $n\equiv2(\modu4)$&\begin{tikzpicture}
[->,>=stealth',shorten >=1pt,auto,node distance=1.3cm,thick,main node/.style=]
   \node[main node] (1) {${j_1}$};
  \node[main node] (2) [right of=1]{$j_2$};
  \node[main node] (3) [right of=2]{$\cdots$};   
  \node[main node] (4) [right of=3] {$j_{l_j}$};  
  \node[main node] (5) [right of=4]{$j_{l_j+1}$};
\path[every node/.style={font=\sffamily\small}]     
(1) edge [loop above] node              {} (1) 
(2) edge [loop above] node              {} (2) 
(2) edge [bend left=20] node  {} (1)        
(2) edge [bend right=20] node              {} (1)
(3) edge node                              {} (2)  
(4) edge [loop above] node              {} (4)
(4) edge node {}(3)
(5) edge [loop above] node              {} (5)
(5) edge [bend right=20] node  {} (4)        
(5) edge [bend left=20] node              {} (4)
;
\end{tikzpicture}&$1\leq g\leq \frac{n}{2}$&$0\leq l\leq 1$&$1\leq i\leq 2$\\ \hline
$C_{(n)}$ and $n\equiv0(\modu4)$&\begin{tikzpicture}
[->,>=stealth',shorten >=1pt,auto,node distance=1.3cm,thick,main node/.style=]
  \node[main node] (1) {${j_1}$};
  \node[main node] (2) [right of=1]{$j_2$};
  \node[main node] (3) [right of=2]{$\cdots$};   
  \node[main node] (4) [right of=3] {$j_{l_j}$};  
  \node[main node] (5) [right of=4]{$j_{l_j+1}$};
\path[every node/.style={font=\sffamily\small}]     
(1) edge [loop above] node              {} (1) 
(2) edge [loop above] node              {} (2) 
(2) edge [bend left=55] node  {} (1)        
(2) edge [bend right=55] node              {} (1)
(2) edge [bend left=20] node  {} (1)        
(2) edge [bend right=20] node              {} (1)
(3) edge node                              {} (2)  
(4) edge [loop above] node              {} (4)
(4) edge node {}(3)
(5) edge [loop above] node              {} (5)
(5) edge [bend right=20] node  {} (4)        
(5) edge [bend left=20] node              {} (4)
;
\end{tikzpicture} &$1\leq g\leq \frac{n}{4}$&$0\leq l\leq 3$&$1\leq i\leq 2$\\ \hline
\end{tabular}}
\caption{Let $\A_{j_i\,j_{i+1}}\colon j_i\lra j_{i+1}$ be an arrow  for $1\leq i\leq l_j$. We can see that each path from $m$ to $1$ is $\epsilon_{j_1}^r p^{(j)}_{g\,l\,i}$, $r\in\Z_{\geq0}$, where $p^{(j)}_{g\,l\,i}$ is equal to $\A_{j_1\,j_2}^{(g)}\epsilon_{j_2}^l\A_{j_2\,j_3}\cdots\A_{j_{l_j-1}\,j_{l_j}}\A^{(i)}_{j_{l_j}\,j_{l_j+1}}$.}
\label{table4}
\end{table}
 We have a map $h\colon\C[\![\E^\frac{1}{4}]\!]\lra\C[\![\epsilon_1]\!]$ defined by $\E^\frac{1}{4}\mapsto \epsilon_1$. Without loss of generality, we can  assume that there is a $\B'_{g_0+1\,l_0}\otimes e_{i_0}\in B_1$ such that 
 $$\cogrado(\widetilde{M}_1(\B'_{g_0+1\,l_0}\otimes e_{i_0}))=\min\{\cogrado(\widetilde{M}_j(x))|x\in B_j,\,1\leq j\leq t\}.$$
  Let $\mathcal{P}$ be an ideal generated by:
\begin{enumerate}
\item The set $\{p^{(1)}_{g_0\,l_0\,i_0}+h(\widetilde{M}_j(\B'_{g-1\,l}\otimes e_i))p^{(j)}_{g\,l\,i}|\B'_{g-1\,l}\otimes e_i\in B_j, 1\leq j\leq t\}$, when $\cogrado(\widetilde{M}_1(\B'_{g_0+1\,l_0}\otimes e_{i_0})))\geq0$.
  
\item The set $$\{p^{(1)}_{g_0\,l_0\,i_0}-h(\widetilde{M}_j(\B'_{g-1\,l}\otimes e_i)\widetilde{M}^{-1}_j(\B'_{g_0+1\,l_0}\otimes e_{i_0}))p^{(j)}_{g\,l\,i}|\B'_{g-1\,l}\otimes e_i\in B_j, 1\leq j\leq t\},$$ when $\cogrado(\widetilde{M}_1(\B'_{g_0+1\,l_0}\otimes e_{i_0})))<0$. 

\end{enumerate}

 Now, for these $C_{(l,\,P)}$, $D_{(l,\,P)}$, $\mathcal{I}$ and $\Omega_{\mathcal{I}}$, we want to define  an algebra $H$ as in \cite{GLS16},  but in our case, we have to consider, in some cases, twisted bimodules to produce the relations imposed by the adjoint maps.  For each $a\in \{1,\,2,\cdots,m\}$, we set $H_a\coloneqq \C[\epsilon_a]/(\epsilon_a^{c_a})$ as a truncated polynomial ring.  For $(a,b)\in\Omega_{\mathcal{I}}$, we define the cyclic $H_a$-$H_b$-bimodule $_aH^{\tw}_b$ as follows:
 \begin{eqnarray*}
_aH_b^{\tw}=\left\{\begin{array}{ll} (\A_{ab})/(\epsilon^2_{a}\A_{ab}+\A_{ab}\epsilon_{b})& \mbox{If $a=j_1$ and $C_{I_j}=C_{(1,\,2)}$,}\\ &\\(\A_{ab}^{(1)}, \A_{ab}^{(2)})/(\epsilon_{a}\A_{ab}^{(g)}-(-i)^g\A_{ab}^{(g)}\epsilon_{b})& \mbox{if $a=j_1$ and $C_{I_j}=C_{(2,\,4)}$,}\\&\\  (\A_{ab}^{(g)}|1\leq g\leq g_{ab})/(\epsilon^{f_{ba}}_{a}\A_{ab}^{(g)}-\A_{ab}^{(g)}\epsilon_{b}^{f_{ab}})&\mbox{otherwise.}\end{array}\right.
\end{eqnarray*} 
 Then, we define 
 $$S\coloneqq \prod_{1\leq a\leq m}H_a \,\,\,\,\, \mbox{ and }\,\,\,\,\,B^{\tw}\coloneqq \bigoplus_{(a,b)\in\Omega_{\mathcal{I}}}\,_aH_b^{\tw}.$$
 Note that $B^{\tw}$ is naturally an $S$-$S$-bimodule. Thus, we can define the tensor algebra
 $$H^{\tw}(l,\,P)\coloneqq  H^{\tw}_\C(C_{(l,\,P)},D_{(l,\,P)},\Omega_{\mathcal{I}})\coloneqq  T_S(B^{\tw})/\mathcal{P}.$$

   For each $1\leq a\leq m$, we take $\E_a=\E^\frac{1}{c_a}$. In particular, the field extension $k\subseteq k_{c_a}$ has degree $c_a$.  We set $\widetilde{S}\coloneqq \prod_{1\leq a\leq m} k_{c_a}$ and for  all $(a,b)\in\Omega_{\mathcal{I}}$:
   $$_{a}\widetilde{H}_{b}^{\tw}\coloneqq \left\{\begin{array}{ll}K^{\Si} & \mbox{If }a=j_1\mbox{ and } C_{I_j}= C_{(1,\,2)},\\K^\Si\times K^{\Si^2} & \mbox{if }a=j_1\mbox{ and }\, C_{I_j}= C_{(2,\,4)}\mbox{,}\\ k_{l_{ab}}^{g_{ab}}&\mbox{ otherwise,}\end{array}\right.$$

  with $l_{ab}=\mcm(c_{a},c_{b})$. In particular, $_{a}\widetilde{H}_{b}^{\tw}$ is a $k_{c_{a}}$-$k_{c_{b}}$-bimodule for $(a,\,b)\in \Omega_{\mathcal{I}}$.   Finally, we define the  tensor algebra
  $$\widetilde{H}^{\tw}(l, P)\coloneqq \widetilde{H}^{\tw}_k(C,D,\Omega_{\mathcal{I}})\coloneqq  T_{\widetilde{S}}\left(\bigoplus_{(a,b)\in\Omega_{\mathcal{I}}}\,_a\widetilde{H}_b\right)/\mathcal{P}',$$
  where $\mathcal{P}'$ is generated by:
  \begin{enumerate}
\item The set
  $$\{\overline{\B'_{g_0+1\,l_0}\otimes e_{i_0}}+\widetilde{M}_j(\B'_{g-1\,l}\otimes e_i)\overline{\B'_{g-1\,l}\otimes e_i}|\B'_{g\,l}\otimes e_i\in B_j, 1\leq j\leq t\},$$
  when  $\cogrado(\widetilde{M}_1(\B'_{g_0+1\,l_0}\otimes e_{i_0})))\geq0$.
\item The set 
$$\{\overline{\B'_{g_0+1\,l_0}\otimes e_{i_0}}-\widetilde{M}_j(\B'_{g-1\,l}\otimes e_i)\widetilde{M}^{-1}_j(\B'_{g_0+1\,l_0}\otimes e_{i_0})\overline{\B'_{g-1\,l}\otimes e_i}|\B'_{g\,l}\otimes e_i\in B_j, 1\leq j\leq t\},$$
when  $\cogrado(\widetilde{M}_1(\B'_{g_0+1\,l_0}\otimes e_{i_0})))<0$.
\end{enumerate}
Recall that if $v\otimes u\in B_j$, then $\overline{v\otimes u}=v\otimes\underbrace{1\otimes\cdots\otimes1}_{l_j-2}\otimes u$.
 \begin{teo} Let $l\coloneqq (l_1,\cdots,\,l_t)\in\N_{\geq2}^t$ and $M_1,\,M_2,\cdots,M_t$ be simple regular representations that are pairwise non-isomorphic. Then, there is $P:=\{\rho_1,\cdots,\rho_t\}\subset \spec[k]\dot{\cup}\{1,\,2\}$ such that
 $$\widetilde{H}^{\tw}(l,\,P)=\mathcal{C}(l,\{M_j\}_{1\leq j\leq t}).$$
 \end{teo}
 \begin{proof}
 From Theorem \ref{lema1}, we have, for $1\leq j\leq t$, that $M_j\cong M_{(1)}$, $M_{(2)}$ or $M_{n_j}^{(i)}(s_j)$ for some $n_j\in\Z_{>0}$ and $s_j\in k_{n_j}$. Then, we can choose $\rho_j\in P$, as follows:
 $$\rho_j=\left\{\begin{array}{cl}1&\mbox{ If }M_j\cong M_{(1)},\\2&\mbox{ if }M_j\cong M_{(2)},\\\prod_{l=0}^{n_j-1}(x-\Si_{n_j}^l(\phi_n^{(i)}(s_j)))&\mbox{ If }M_j\cong M_{n_j}^{(i)}(s_j).\end{array}\right.$$
 
This is the $P$ we are looking for.
 \end{proof}
\addcontentsline{toc}{section}{Acknowledgements} \section*{Acknowledgements}
The first author was supported by CONACyT's grant 229255. The second author gratefully acknowledges the support he received from  CONACyT-575165 scholarship. Both authors would like to thank the reviewers for their valuable comments and suggestions, which significantly contributed to improving the quality of this article.
\section*{Statements and Declarations}
{\bf Ethical Approval:} not applicable.

{\bf Funding:} The first author was supported by CONACyT's grant 229255. The second author gratefully acknowledges the support he received from  CONACyT-575165 scholarship. 

{\bf Availability of data and materials:} not applicable.

{\bf Conflict of interest statement:}  Not applicable.  
\addcontentsline{toc}{section}{References} %\section*{References}

\end{document}